\documentclass[amsart]{article}

\usepackage{latexsym}
\usepackage{amssymb}
\usepackage{amstext}
\usepackage{amscd}
\usepackage[dvips]{graphicx}
\usepackage{graphicx}
\usepackage{amsthm}
\usepackage{amsfonts}	
\usepackage{amsmath}

\newtheorem{teo}	{Theorem}[section]
\newtheorem*{mteo}	{Main Theorem}
\newtheorem*{teoa}	{Theorem A}

\newtheorem{defi}	{Definition}

\newtheorem{pr}		{Proposition}[section]
\newtheorem{lm}		{Lemma}[section]
\newtheorem{cor}	{Corollary}[section]

\newtheorem{rem}	{Remark}

\newcommand{\N}		{\mathbb{N}}
\newcommand{\Z}		{\mathbb{Z}}

\newcommand{\R}		{\mathbb{R}}
\newcommand{\D}		{\mathbb{D}}
\newcommand{\C}		{\mathbb{C}}
\newcommand{\Cb}	{\overline{\mathbb{C}}}
\newcommand{\Ctwo}	{{\mathbb{C}}^2}

\renewcommand{\P}	{\mathbb{P}}

\newcommand{\Hc}	{{\mathring{H} }}

\newcommand{\bpl}	{\bigl\{ }
\newcommand{\bpr}	{\bigr\} }

\newcommand{\ww}  [2]	{W^{#1}({#2})}
\newcommand{\www} [3]	{W^{#1}_{#3}({#2})}

\newcommand{\graph}	{{\rm graph\, }}
\newcommand{\crit}	[1]	{{\rm Crit(#1)\, }}
\newcommand{\cv}	[1]	{{\rm CVal(#1)\, }}

\newcommand{\reg}	{\mathcal{R}}

\newcommand{\Isom}	{{\rm Isom}}
\newcommand{\supp}	{{\rm supp\, }}
\renewcommand{\exp}	{{\rm exp}}
\renewcommand{\mod}	{{\rm mod}}

\begin{document}

\title{\bf On Critical Point for Two Dimensional Holomorphics Systems}

\author{Francisco Valenzuela--Henr\'{\i}quez\thanks{Part of this work was supported by CNPq--IMPA, by PEC--PG, and by Proyecto Interno Postdoctorado 2011 of PUCV.} \footnote{Instituto de Matem\'atica, Pontificia Universidad Cat\'olica de Valpara{\'{\i}}so, Blanco Viel 596, Cerro Bar\'on, Valpara{\'{\i}}so, Chile.}\\
\texttt{\small{francisco.valenzuela@ucv.cl}}}

\date{}

\maketitle

\begin{abstract}
Let $f:M\rightarrow M$ be a biholomorphisms on two--dimensional a complex manifold , and let $X\subseteq M$ be a compact $f$--invariant set such that $f|X$ is asymptotically dissipative and without sinks periodic points. We introduce a solely dynamical obstruction to dominated splitting, namely critical point. Critical point is a dynamical object and capture many of the dynamical properties of their one--dimensional counterpart.
\end{abstract}

\tableofcontents

\section{Introduction}\label{sec:intro}

Interested in a global understanding of holomorphic dynamics in higher dimensions, the polynomial automorphisms of  $\C^2$ known as generalized H\'enon maps (or complex H\'enon maps), appear in a natural way in this study.

These maps were a subject
of serious study with foundational work carried out by J. Hubbard \cite{h}, J. Hubbard and R. Oberste--Vorth \cite{ho1,ho2}, E. Bedford and J. Smillie
\cite{bs1,bs2,bs3}, E. Bedford, M. Lyubich and J. Smillie \cite{bls1}, S. Friedland and J. Milnor in \cite{fm}, J. Forn\ae{}ss and N. Sibony \cite{fs1}, among others.

Complex H\'{e}non maps are obtained as a finite composition of maps
of the form $(y,p(y)-bx)$ where $p$ is a polynomial of degree at least two and
$b \in \Bbb{C}^*$. As in the case of rational maps in the one--dimensional context, complex H\'enon maps have a well--defined Julia set $J$. 
This set is a compact invariant set that contains the support of the unique measure of maximal entropy. (See \cite{bs1}.) We will denote the support of the measure of maximal entropy  by $J^*$.

A significant open question in the study of complex H\'{e}non maps is whether $J = J^*$. Bedford and Smillie~\cite{bs1} show that if $f$ is uniformly
hyperbolic on $J$, then $J=J^*$. Further, Forn\ae{}ss~\cite{F} shows that if $f$ is uniformly hyperbolic on $J^*$, and  is not volume preserving,
then $J=J^*$. In the setting of complex H\'enon maps, hyperbolicity is the natural generalization of the expansiveness on the Julia set for rational maps
on $\Cb$. 

We recall that dominated splitting is a weaker version of hyperbolicity. In this setting, it is reasonable to look the property of dominated splitting, to verify the possible hyperbolicity. Recently in our work \cite{pancho01}, was given several equivalences of uniform hyperbolicity, under the hypothesis of dominated splitting in the Julia set. Then a natural question appear:

\smallskip

\noindent
{\bf Question:}\, \, {\it When the set $J/J\sp*$  has dominated splitting?.}

\smallskip

In the context of real and complex one--dimensional mapping, this phenomena is already known: for one--dimensional endomorphism, one obstruction to hyperbolicity is the presence of critical points (points with zero derivative) in the limit set. In fact, in the real context, Ma\~n\'e showed that smooth and generic (Kupka--Smale) one--dimensional endomorphisms without critical points are either hyperbolic or conjugate to an irrational rotation (see \cite{ma2}). So we could say that, for generic smooth one--dimensional endomorphisms, any compact invariant set is hyperbolic if, and only if, it does not contain critical points.
In the complex  case (rational maps), it is well known that the Julia set $J$ is hyperbolic if, and only if, $J$ is disjoint of the post critical set. We recall that in dimension one, hyperbolicity and dominated splitting is the same notion.

{\bf Our main goal:} is to introduce the dynamical obstruction to dominated splitting for two--dimensional biholomorphisms in a complex manifolds. This allow us to introduce one notion of {\bf critical point} for complex H\'enon maps, that capture many of the dynamical properties of their one--dimensional counterpart.

One notion of critical point on surfaces, was introduced by E. Pujals and F. Rodrigues Hertz in their work \cite{P-RH}. They works with systems that are dissipative in a compact invariant set without sink periodic point. The main result of \cite{P-RH} state that $C\sp2$--generically a systems has dominated splitting, if and only if, the set of critical point is empty. From Theorem B of Pujals--Sambarino in \cite{ps}, the authors of \cite{P-RH} conclude that: {\it Generically, an invariant set is either an hyperbolic set or an normally hyperbolic closed curve which dynamics is conjugate to an irrational rotation if and only if the set of critical points is empty}. We remark that in \cite{P-RH}, the authors performs the proof of their main result, using Theorem B on \cite{ps}. Later, S. Crovisier in \cite{C}, was give one prove of the main result on \cite{P-RH}, independent of the  Pujals--Sambarino's Theorem.

We will make a first presentation of our main result, in the context of complex H\'enon maps. To introduce the notion of critical point we look for the projective action of the derivative of the map. More precisely, let $f$ be a dissipative complex H\'enon map, i.e. $|\det(df_x)|=b<1$. Let $F_x$ be the M${\ddot{\rm o}}$bius transformation induced by $df_x$. We denote the spherical norm of the derivative of $F_x$ at the point $\xi\in\Cb$ by $||F_x\sp\prime(\xi)||$.

Let $b<\beta_+\le\beta_-<1$ and $\beta=(\beta_-,\beta_+)$. We say that $x\in J$ is a {\bf $\beta$--critical point} if there exists a direction $\xi$ such that $||(F\sp{\pm n}_x)\sp\prime(\xi)||\ge \beta_\pm\sp{\pm n}$, for each $n\ge0$. We denote the set of all $\beta$--critical points by $\crit\beta$. The preceding definition assert that a point is critical, if there exist a (projective) direction that is expanded (in norm) to the past by the action of $F$, and is not very contracted to the future. We recall that this definition generalize in the complex case, the notion of critical point in \cite{P-RH}, that turn, and quoting the words of \cite{P-RH} authors, {\it ``\ldots (a critical set)\ldots goes back to the seminal studies done for H\'enon attractor in \cite{B-C}''}. Also recall that this definition is adapted to the dissipative context.

Our main result can be now stated as follows:

\begin{mteo}
Let $f$ be a dissipative complex H\'enon map, with  $|\det(df_x)|=b<1$. Then $J$ has dominated splitting if and only if for every $\beta=(\beta_-,\beta_+)$ where $b<\beta_+\le\beta_-<1$, the set $\crit\beta$ is an empty set.
\end{mteo}

This Theorem is consequence of a more general version of this result, stated for complex linear cocycles, namely {\bf Theorem A}. An hypothesis necessary both in the surfaces, and Theorem A version, is the absence of sinks. Since Julia set only contain periodic saddle points, this hypothesis not appear in the statement of Main Theorem. For now, there are not a Pujals--Sambarino's Theorem in the two--dimensional complex case. The way to prove Theorem A is adapt in our context the main ideas of Crovisier on \cite{C}. However, since the definition of critical point in \cite{P-RH}, and our definition are distinct, the adaptation of this ideas, have several differences with the original version.

In this point, we explain several properties relatively to the critical set (See Subsection 6.2). Firstly, the critical set is a compact set. Introducing a partial order in the set of indexes $\beta$  (we say that $\beta\ge\alpha$ if and only if $\beta_\pm\sp{\pm1}\ge\alpha_\pm\sp{\pm1}$), we have the monotony contention: if $\beta\ge\alpha$ then $\crit\beta\subseteq\crit\alpha$. Also, under change of (hermitian) metric and conjugation, critical point are preserved, maybe after of a finite but bounded iterates to the past or to the future. The critical set is far from dominated/hyperbolic sets and Pesin's Blocks. We recall that in the polynomial case, positive iterates of a critical point can still be critical, but not all element of the whole positive orbit is a critical point. This property also holds for our context: in the orbit of a critical point, there exists a {\it ``distinguished''}\,  critical point that is the last critical point (last in the point of view of the positive orbit). Critical point is not a regular point in the Osceledets sense. Finally, orbit of a tangencie between the stable and the unstable manifold of a periodic point, contain a critical point.

The paper is organized as follows:

In Section 2, we state result and give the tools related to complex linear cocycles (the action of $df$ in the tangent of $J$)  and projective cocycles (the action of $F$ in the spherical bundle of $J$).

Section 3 is devoted to state the notion of dominated splitting for linear cocycles. We give several tools in terms of the projective cocycle, that are equivalent to the existence of dominated splitting.

In Section 4, we define formally the notion of critical point and state the general version of our Main Theorem (Theorem A). Also we state the notions necessaries to prove our Theorem A.

Section 5, is devoted to proving Theorem A.

Section 6, we study several properties of critical points. More over, in subsection 6.5, we conjecture (for complex H\'enon maps) the existence of another ``critical points'' outside of $J$, in order to obtain a two dimensional counterpart for the classical one dimensional Theorem  about rational maps:\, {\it If the post--critical is disjoint from the Julia set, them the Julia set is hyperbolic.}

\bigskip

\noindent {\it Acknowledgements.--}\, This paper is part of my thesis, IMPA 2009. I want to thank my advisor Enrique R. Pujals for his significant orientation and his constant encouragement while carrying out thesis work. I am
particularly grateful to IMPA for its support during my doctorate studies. While preparing this work I enjoyed the hospitality of both the Pontificia
Universidad Cat\'olica de Valpara\'{\i}so and the Universidad Pontificia Universidad Cat\'olica de Chile in Santiago city. I wish to give a special thank to Juan Rivera Lettelier (PUC) for several suggestions in order to obtain a final definition of critical point.

\section{Preliminaries}\label{s:preliminaries}

\subsection{Bundles and Cocycles}\label{ss:bundles}

Let $X$ be a topological compact space. We let $TX$ denote a locally trivial vector bundle of complex dimension 2 over $X$. We denote by $T_z$ the fiber of $TX$ at $z\in X$, and denote by $pr:TX\rightarrow X$ the natural projection.
A {\bf linear cocycle} $A:TX\rightarrow TX$ is a continuous isomorphism in the category of the vector bundles. More precisely, $A$ is continuous and there exists $f:X\rightarrow X$ an homeomorphism such that $A_z=A|T_z:T_z\rightarrow T_{f(z)}$ is a complex isomorphisms. We say that the homeomorphism
$f$ is the {\bf base} of the cocycle $A$.

\smallskip

Given a vector bundle $PX$ we define its {\bf projective
bundle} as the set $\P(X)=\cup_{z\in X}\bpl
z\bpr\times\C\P^1(T_z)$. The projective bundle is a bundle which fiber is the Riemann sphere $\Cb$. Denote the canonical projections by $p:TX^*\rightarrow\P(X)$, where $TX^*=TX\setminus\bpl \textrm{the zero section}\bpr$. In what follows, we denote by $\Cb_z$ the fiber of $\P(X)$ at $z\in X$. Given $A$ a linear cocycle with base
$id$, $p_A$ denotes the map from $TX^*$ to $\P(X)$ satisfying $p_A|_{T_z^*}=p\circ A_z$.

\smallskip

Similarly as in the vector bundle case, we say that $M:\P(X)\rightarrow\P(X)$ is a {\bf projective cocycle}, if it is a continuous isomorphisms in the category of bundles with projective fibre. Given a linear cocycle $A$ we can associate to it a projective cocycle $M$ in a natural fashion as $M_z([v])=[A_zv]$, where $[\cdot]$ denote the equivalence class in the projective space.

\smallskip

Given a non--negative integer $l$ we define the iterate of the cocycle $A$ by the equation
\[A^l_z=A_{f^{l-1}(z)}\circ\cdots\circ A_{f(z)}\circ A_z\, \textrm{
and }\, A^{-l}_z=A^{-1}_{f^{-l}(z)}\circ\cdots\circ
A^{-1}_{f^{-2}(z)}\circ
A^{-1}_{f^{-1}(z)},\] and define $A\sp0=Id$ by convention. In the same way we define the iterates $M\sp l_z$ and $ M^{-l}_z$ for the projective cocycle.

\smallskip

We have well--defined an hermitian (a spherical) metric in the linear (projective) bundle. Let $TX\odot TX$ be the subset of $TX\times TX$ consisting of pairs $(u,v)$ such that $u$ and $v$ are in the same fiber. An {\bf hermitian metric} on $TX$ is a continuous function $(\cdot|\cdot):TX\odot TX\rightarrow\C$ such that $(\cdot|\cdot)|T_z\times T_z=(\cdot|\cdot)_z$ is an hermitian product in $T_z$. Since $X$ is compact, there exists an hermitian metric on $TX$ (cf. \cite{H}). In what follows, we denote $||v||_z=(v|v)_z$.

\smallskip

We also have the following statement.

\smallskip

\begin{defi}
The {\bf spherical metric} in the projective bundle $\P(X)$, is the metric induced by the hermitian metric in $TX$.
\end{defi}

\smallskip

For see the formal construction of the previous definition, we suggest to the reader consult the Appendix A.

\subsection{Oseledets Theorem}

We say that a point $z\in X$ is a {\bf regular point of $A$}, if the fiber $T_z$ admits a splitting $T_z=E_z\oplus F_z$ of one dimensional complex subspaces, and numbers $\lambda^-(x)\leq\lambda^+(z)$ satisfying
\[\lim_{n\rightarrow\pm\infty}\frac{1}{n}
\log(||A^n_zu||)=\pm\lambda^-(z)\, \, \, \,
\textrm{and}\, \, \, \, \lim_{n\rightarrow\pm\infty}\frac{1}{n}
\log(||A^n_zv||)=\pm\lambda^+(z),\]
where $u\in E_z\setminus\bpl 0\bpr$ and $v\in F_z\setminus\bpl 0\bpr$. Recall that a set $S\subset X$ has {\bf total probability in $X$}, if for every $f$-invariant measure $\mu$, we have $\mu(S^c)=0$.

\smallskip

\begin{teo}{\bf (Oseledets)}
The set of regular points of $A$ has total probability. Moreover, $z\mapsto E_z$ and $z\mapsto F_z$ are measurable subbundles and the functions $z\mapsto\lambda^\pm(z)$ are measurable.
\end{teo}

\smallskip

For a proof of Oseledets's Theorem in the setting of
cocycles, see \cite{V}.

\smallskip

We denote the set consisting of all regular points of a cocycle $A$ by $\mathcal{R}(A)$. The Oseledets's Theorem asserts that given an $f$--invariant measure $\mu$, the set of regular points in the support of $\mu$ has total measure. Indeed, we have that $\mu(\mathcal{R}(A)\cap\supp(\mu))=1$.
We denote the set $\mathcal{R}(A)\cap\supp(\mu)$ by $\mathcal{R}(A,\mu)$.

\smallskip

In the original work of Pujals and Rodriguez Hertz (see \cite{P-RH}), an important hypothesis is the absence of saddle periodic points. In our setting, we replace this hypothesis for our next notion.

\smallskip

\begin{defi}\label{def:partially_hyp_measure}
We say that a measure $f$--invariant $\mu$ is {\bf partially hyperbolic}, if for any $x\in\mathcal{R}(A,\mu)$ the inequality $\lambda^-(x)<0\leq\lambda^+(x)$ hold.

We also say that $f$ {\bf has no attractors} (in the broad sense) if all $f$--invariant measure is partially hyperbolic.
\end{defi}

\subsection{The Multiplier}

In the studies of rational maps in the Riemann sphere, an
important tool to describe the dynamics near a  periodic
point, is the notion of multiplier. By B$\rm{\ddot{o}}$cher's
Theorem, the dynamic in a neighborhood of the periodic point is (via
conjugation) given by the dynamics of the map $w\mapsto \lambda w$,
where $\lambda$ is called the multiplier of the point. In many cases, we are
interested in the norm of the multiplier.

\smallskip

For a point $z\in\Cb$ which is not periodic, it is possible to define a similar tool as the multiplier, using the spherical metric.

\smallskip

\begin{defi}\label{def:g}
Let $U\subset \Cb$ be an open set and $R:U\rightarrow\Cb$ be an holomorphic map. We define the {\bf norm of the multiplier of $R$ at the point $z$}, as the spherical norm of the derivative of $R$ at the point $z$. That is,
\begin{equation}\label{g_def}
||R\sp\prime(z)||=\sup\left\{\frac{||R\sp\prime(z)\xi||_{R(z)}}{||\xi||_z}\, :\, \xi\in T_z\Cb\right\},
\end{equation}
where $||\cdot||_z$ denote the spherical norm in $T_z\Cb$.
\end{defi}

\smallskip

Under the identification $T_z\Cb$ with $\C$, an explicit expression for the spherical metric is
\begin{equation}\label{sph_met}
||\xi||_z=\frac{2|\xi|}{1+|z|\sp2}.
\end{equation}
Thus, it is not difficult to see that
\begin{equation}\label{ec:03}
||R\sp\prime(z)||=|R\sp\prime(z)|\cdot\frac{1+|z|\sp2}{1+|R(z)|\sp2}.
\end{equation}

\smallskip

An important result is the following.

\smallskip

\begin{pr}\label{pr:multip}
Let $z_0\in\Cb$ and $R$ be a rational map. Also let $f$ and $h$ two isometries in the Riemann sphere such that $f(z)=h(R(z))=0$. Then
\[||R\sp\prime(z)||=|(h\circ R\circ f\sp{-1})'(0)|.\]
\end{pr}
\begin{proof}
Since $f$ is an isometry in the Riemann sphere, for each $v\in T_w\Cb$ we have that $||v||_w=||f\sp\prime(w)v||_{f(w)}$, and the same equality holds for $h$.

\smallskip

Thus we have that,
\[\frac{2}{1+|z|\sp2}=||1||_z=||f\sp\prime(z)||_0=2|f\sp\prime(z)|\]
and
\[\frac{2}{1+|R(z)|\sp2}=||1||_{R(z)}=||h\sp\prime(R(z))||_{0}=2|h\sp\prime(R(z))|,\]
therefore we conclude that
\[|(h\circ R\circ f\sp{-1})'(0)|=|(h\sp\prime(R(z))|\cdot|R\sp\prime(z)|\cdot| f\sp\prime(z)|\sp{-1}=\frac{1}{1+|R(z)|\sp2}\cdot|R\sp\prime(z)|\cdot(1+|z|\sp2).\]
\end{proof}

\smallskip

In the following lemma, we give an explicit formula to calculate the norm of the multiplier, for a M${\ddot{\rm o}}$bius transformation.

\smallskip

\begin{lm}\label{lemma:g} Let $M$ be a M$\rm{\ddot{o}}$bius transformations given by
\[M(\xi)=\frac{a\xi+b}{c\xi+d},\]
then
\begin{equation}\label{definition_g_2}
||M\sp\prime(z)||=\frac{|\delta|}{||Av_z||^2}
\end{equation}
where
\[A=\left(\begin{array}{cc}
   a&b\\
   c&d
  \end{array}\right),
\]
$v_z$ is an unitary vector in $\Ctwo$ whit
$[v_z]=z$, and $\delta=\det(A)$.
\end{lm}
\begin{proof}
From equation (\ref{ec:03}),  we have that
\[||M\sp\prime(z)||=|M'(z)|\cdot\frac{1+|z|^2}{1+|M(z)|^2}=
\frac{|\delta|}{|cz+d|^2}\frac{1+|z|^2}{1+|M(z)|^2}.\]
If we take $v_z=(v^1_z,v^2_z)\in\C^2$ an unitary vector such that $z=v^1_z/v^2_z$, then
\[\begin{aligned}
||M\sp\prime(z)||&=|\delta|\cdot\frac{1+|z|^2}{|az+b|^2+|cz+d|^2}\\&=
|\delta|\cdot\frac{1+\left|\frac{v^1_z}{v^2_z}\right|^2}{\left|a\frac{
v^1_z}{v^2_z}
+b\right|^2+\left|c\frac{v^1_z}{v^2_z}+d\right|^2}\\
&=|\delta|\cdot\frac{|v^1_z|^2+|v^2_z|^2}{
|av^1_z+bv^2_z|^2+|cv^1_z+dv^2_z|^2}\\
&=\frac{|\delta|}{||Av_z||^2}.
  \end{aligned}
\]
\end{proof}

\smallskip

The equation (\ref{definition_g_2}), hallow define the norm of a multiplier for a projective cocycle $M$ in terms of his linear cocycle related with them, and the hermitian metric considered in the fibre bundle.

\smallskip

\begin{defi}\label{defi:g_for_cocyles}
Let $A$ be a linear cocycle on $TX$ and $M$ the natural projective cocycle related with them. Given $\xi\in\P(X)$ with $\xi\in\Cb_z$ we define the {\bf norm of the multiplier of $M$ at the point $\xi\in\Cb_z$} by
\begin{equation}\label{g_1_iterate}
g(1,\xi)=g(1,\xi,M)=\frac{|\det(A_z)|}{||A_zv_\xi||_{f(z)}^2}
\end{equation} and similarly, the norm of the multiplier of $M\sp n$ at the point $\xi$ is defined by
\begin{equation}\label{g_n_iterate}
g(n,\xi)=g(1,\xi,M\sp n)=\frac{|\det(A\sp n_z)|}{||A_z\sp
nv_\xi||_{f\sp n(z)}\sp 2}
\end{equation}
where $v_\xi$ is choose unitary and such that $[v_\xi]=\xi$.
\end{defi}

\smallskip

\begin{rem}
It follows from equation (\ref{g_n_iterate}) that $g(n+m,\xi)=g(n,M\sp n\xi)\cdot g(m,\xi)$ for each $n,m\in\Z$. This prove is elementary and will be left to the  reader.
\end{rem}

\smallskip

In the following subsection we justify the Definition \ref{def:g}, but is not essential for the rest of this work,  and may be skipped.

\subsection{Some Remarks}
In the remainder of this section, we explain with detail the motivations for Definition \ref{def:g}. The experienced reader can skip directly to the next section.

\smallskip

First, we recall the formal definition of {\bf multiplier} for a rational map at a fixed point.

\smallskip

\begin{defi}[{\bf Multiplier}]\label{def:multiplier}
Let $R$ be a rational function on the Riemann sphere $\Cb$, and let $z\in\Cb$ be a fixed point of $R$.

\begin{itemize}
\item[{\it i)}] If $z\in\C$ we define the
multiplier of $R$ in the point $z$ by $R^\prime(z)$, and is denoted by
$\lambda(z,R)$.
\item[{\it ii)}] If $z=\infty$ we choose a M$\it{\ddot{o}}$bius map
$f$ such that $f(\infty)\in\C$, and it is defined
$\lambda(\infty,R)=\lambda(f(\infty),f\circ R\circ f^{-1})$.
\end{itemize}
\end{defi}

\smallskip

Note that in the preceding definition the value of $\lambda(z,R)$ when $z\in\C$ remains {\it invariant under conjugation} by a M$\rm{\ddot{o}}$bius transformations. It follows that $\lambda(\infty,R)$ is well--defined.

\smallskip

Remember that a M$\rm{\ddot{o}}$bius transformation $T$ is an isometry in the Riemann sphere whit the spherical metric, if and only if we can write
\[T(w)=T_{a,b}(w)=\frac{\overline{b}w+a}{-\overline{a}w+b},\]
with $a$ and $b$ complex number and $|a|^2+|b|^2=1$. Note that if we write $z=a/b$ (and $z=\infty$ if $b=0$) then $T(0)=z$. Since $T$ is an isometry, we conclude that $T(\infty)=z^*$ where $z\sp*$ is the antipodal point of $z$, that is $z\sp*=-1/\overline{z}$. Denote the set consisting of all isomorphisms in the Riemann sphere with the spherical metric by $\Isom(\Cb)$.

\smallskip

In this point, we want to extend the notion of multiplier for the case in which $z$ is not a fixed point. The next proposition stay that this extension can not be done as expected.

\smallskip

\begin{pr}\label{pr_multip}Let $D\subset \Cb$ be a topological disc
and $R:D\rightarrow\Cb$ be an holomorphic map. Let $z\in D\mapsto
f_{i,z},h_{i,z}\in \Isom(\Cb)$ with $i=1,2$, continuous
maps such that $f_{i,z}(z)=h_{i,z}(R(z))=0$.
If we define $F_{1,z}(w)=h_{1,z}\circ R\circ f_{1,z}^{-1}(w)$ and
$F_{2,z}(w)=h_{2,z}\circ R\circ f_{2,z}^{-1}(w)$ in some neighborhood
of zero, then there exists a unique continuous function
$\xi:D\rightarrow\mathbb{S}^1$ such that
\begin{equation}\label{mult_bruto}
 F_{1,z}^\prime(0)=\xi(z)F_{2,z}^\prime(0).
\end{equation}
\end{pr}
\begin{proof}
If we write $f_{i,z}=T_{a_i(z),b_i(z)}$ and
$h_{i,z}=T_{c_i(z),d_i(z)}$, then
$f_{1,z}\circ f_{2,z}^{-1}(w)=\xi_1(z)\cdot w$ and
$h_{2,z}\circ
h_{1,z}^{-1}(w)=\xi_2(z)\cdot w$, where $\xi_1(z)=\zeta_1(z)/
\overline{\zeta}_1(z)$, $\xi_2(z)=\zeta_2(z)/\overline{\zeta}_2(z)$,
$\zeta_1(z)=a_1(z)\overline{a}_2(z)+\overline{b}_1(z)b_2(z)$, and
$\zeta_2(z)=c_2(z)\overline{c}_1(z)+\overline{d}_2(z)d_1(z)$. Taking
$\xi(z)=(\xi_1(z))^{-1}\cdot\xi_2(z)$ we obtain the Proposition.
\end{proof}

\smallskip

Note that the number $\widetilde{\lambda}(z,R):=F_{1,z}^\prime(0)$ is the multiplier in the fixed point case. Nevertheless, the preceding Proposition establishes that $\widetilde{\lambda}(z,R)$ depend of the isometries considered in the ``conjugation'', but $|\widetilde{\lambda}(z,R)|$ is independent. So, is reasonable to call this number as ``the norm of the multiplier''. Since Proposition \ref{pr:multip} establishes that that $|\widetilde{\lambda}(z,R)|$ is equal to the spherical norm of $R\sp\prime(z)$, we decided to accept this terminology in the Definition \ref{def:g}.

\section{Dominated Splitting and Hyperbolic Projective Cocycles}

The main goal of this section is to characterize the notion of dominated splitting for a linear cocycle, in terms of his action in the projective bundle. We introduce the notion of hyperbolic projective cocycle, that roughly speaking, are those cocycles that present the same dynamics as a hyperbolic M${\ddot{\rm o}}$bius transformation. In Theorem  \ref{pr:ds_vs_hyp}, we prove that a linear cocycle has dominated splitting if and only if his projective cocycle is hyperbolic. Moreover, in the same Theorem we estate that the continuity of the section it is not necessary to obtain domination.

\smallskip

We recall the notion of dominated splitting for linear cocycles.

\smallskip

\begin{defi}
We say that a cocycle $A:TX\rightarrow TX$ has {\bf dominated splitting} if there exist an $A$--invariant splitting $TX=E\oplus F$ where $E$ and $F$ are one-dimensional complex planes, and $\,l\ge1$ such that
\[||A_z^l|_{E_z}||\cdot||A^{-l}_{f^l(z)}|_{F_{f^l(z)}}||<\frac{1}{2}.\]
\end{defi}

\smallskip

Recall that if $A$ has dominated splitting, then the $A$--invariant splitting $TX=E\oplus F$ is continuous.
The following is a classical result that establishes equivalences
properties with dominated splitting.

\smallskip

\begin{pr}\label{equivalences_to_ds}Let $A$ be a linear cocycle on a vector bundle $TX$. Then the following statements are
equivalent:
\begin{itemize}
\item[1.] The cocycle $A:TX\rightarrow TX$ has dominated splitting.
\item[2.] There exist an $A$--invariant splitting $TX=E\oplus F$ where $E$ and $F$ are one--dimensional complex planes, a constant $0<\lambda<1$, and a $C>0$ such that
\[||A_z^n|_{E_z}||\cdot||A^{-n}_z|_{F_{f^n(z)}}||\leq C\lambda^n\]for every $z\in X$ and all $n>0$.
\item[3.] There exist a splitting $TX=E\oplus F$ where $E$ and $F$ are one--dimensional complex planes, a constant $l\ge0$, and cone fields $K(\alpha,E)$ and  $K(\beta,F)$, where
\[K(\alpha,E_z)=\bpl u+v\in E_z\oplus F_z\,:\,||u||\leq
\alpha||v||\bpr\]
and
\[K(\beta,F_z)=\bpl u+v\in E_z\oplus F_z\,:\,||v||\leq
\beta||u||\bpr,\]
$A\sp l$--invariant, that is
\[A^{-l}_{f^l(z)}(K(\alpha,E_{f^l(z)}))\subset K(\alpha,E_z)^\circ,\,
\, \, A^{l}_{z}(K(\beta,F_z))\subset K(\beta,F_{f^l(z)})^\circ\]
and
\[||A_z^l|_{K(\alpha,E_z)}||\cdot||A^{-l}_{f^l(z)}|_{K(\beta,F_{f\sp
l(z)})}||<\frac{1}{2}\footnote{This condition appear for real cocycle
of each dimension. In
our case, complex two-dimensional dominated splitting is not
necessary. To see this fact, see Proposition
\ref{contractive_section}, and proof of Theorem
\ref{pr:ds_vs_hyp}, Claim 1, in the Subsection \ref{ss:proof}.}\]
where $K^\circ={\rm int}(K)\cup\bpl 0\bpr$. 
\end{itemize}
\end{pr}

\smallskip

\begin{rem}\label{rem:ds}
\begin{itemize}
\item[1.] It follows from the preceding definition that a linear cocycle $A$ has dominated splitting if and only if $A\sp n$ has dominated splitting, for all $n\ge2$.
\item[2.] Note that item 3 in the preceding establish that it is possible to get the invariant splitting given in item {1} and item 2, by the expressions
\end{itemize}
\begin{equation}\label{invariant_splitting}\widetilde{E}_z=\bigcap_{
n\geq0}A^{-n}_{f^n(z)}
(K(\alpha,E_{f^n(z)}))\, \, \, \, \textrm{and}\, \,
\, \,
\widetilde{F}_z=\bigcap_{n\geq0}A^{n}_{f^{-n}(z)}(K(\beta,F_{f^{-n}(z)
})).
\end{equation}
\end{rem}

\smallskip

The main idea, is note that the invariants splitting determines (in each fibre) two specials points in the sphere, and the cones fields represent disks in the riemann sphere that are asymptotically contracted/expanded by the projective cocycle. We explain this remark in the subsequent paragraphs.

\smallskip

\begin{defi}
\begin{enumerate}
\item We say that a section $\sigma\in\Gamma(X,\P(X))$ is {\bf $M$--invariant} if
$M(\sigma(z))=\sigma(f(z))$.

\item We say that a section $\sigma$ is {\bf contractive} if it is $M$--invariant and there exist constant $C>0$ and a $0<\lambda<1$, such that $g(n,\sigma(z))\leq C\lambda^n$ for every $z\in X$ and all $n\geq 1$. Similarly, we say that a section is {\bf expansive} if this is contractive for $M^{-1}$.

\item We say that a cocycle $M$ is {\bf hyperbolic} if there exist two disjoint sections $\tau$ and $\sigma$ in $\Gamma(X,\P(X))$ (i.e., $\tau(z)\neq\sigma(z)$ for every $z\in X$) such that $\tau$ is expansive and $\sigma$ is contractive.
\end{enumerate}
\end{defi}

\smallskip

So we can state the following:

\smallskip

\begin{teo}\label{pr:ds_vs_hyp}
Let $A$ be a linear cocycle on $TX$ and $M$ the projective
cocycle on $\P(X)$ related with them. Then the following properties are equivalents:
\begin{enumerate}
\item The cocycle $A$ has dominated splitting
\item The cocycle $M$ is hyperbolic.
\item There exists $\sigma$ a contractive section for $M$ (equivalently there exists $\tau$ an expansive section).
\item\label{basical_domination} There exist $C>0$ and $\lambda>1$ such that for every $z\in X$ there exists one direction $\tau_z\in\Cb_z$ such that $g(n,\tau_z)\geq C\lambda^n$ for every $n>0$ (equivalently $\sigma_z\in\Cb_z$ such that $g({-n},\sigma_z)\geq C\lambda^n$ for every $n>0$).
\end{enumerate}
\end{teo}

\smallskip

Theorem \ref{pr:ds_vs_hyp} will be proved in Subsection
\ref{ss:proof}. In the following subsections we explain a series of results necessaries for his proof.

\subsection{Conjugation of Cocycles and Global Trivialization Bundle}

A well known fact about holomorphics maps, is that topological (metrical) contraction of small disc around some point implies that the norm of its derivative is smaller than one and therefore also its multiplier is smaller than one. Since that projective cocycle is holomorphic in each fiber, to determinate if the norm of the multiplier is less to one in some point, it is suffices to determinate if this contract disc around this point. For that, it is natural to look for more simples cocycles which are conjugated to the initial one, and check if the new cocycle shrinks discs. The formal
notion of conjugation is the following definition.

\smallskip

\begin{defi}
Let $M,N:\P(X)\rightarrow\P(X)$ be two cocycles with
$M=(f,M_*)$ and $N=(g,N_*)$. We say that $M$ and $N$ are {\bf conjugated} if there exists a cocycle $H=(h,H_{*}):\P(X)\rightarrow\P(X)$ where $h:X\rightarrow X$ is an homeomorphism such that $H\circ M=N\circ H$.
\end{defi}

\smallskip

The preceding definition state that we have simultaneously the conjugations $hf(z)=gh(z)$ and $H_{f(z)}M_z(\xi)=N_{h(z)}H_z(\xi)$.

\smallskip

In what follows, we will only work with projective cocycles with an invariant (global) section. Whit this hypothesis, the following result establish that the bundle is trivial.

\smallskip

\begin{pr}
If $\,\P(X)$ has a global section, then $\P(X)$ is isometrically equivalent to the trivial bundle $X\times\Cb$.
\end{pr}
\begin{proof}
Let $\sigma\in\Gamma(X,\P(X))$ be a global section and $E$ a splitting in $TX$ associated with this direction. Let us take $\sigma^*$ the global section associated with the direction $E^\perp$, then $\sigma^*(z)$ is the antipodal point of $\sigma(z)$ in the sphere $\Cb_z$.

\smallskip

\noindent
{\bf Claim:} {\it  For every $z\in X$ there exists a biholomorphism $H_z:\Cb_z\rightarrow\Cb$ such that is an isometry, $H_z(\sigma(z))=0$ and $H_z(\sigma^*(z))=\infty$.}

\smallskip

\noindent
{\bf Proof of Claim.} First, let $\bpl (U_i,\varphi_i)\, : \, i=1,\ldots,n\bpr$ be a family of local charts of bundle such that $X=\cup_iU_i$. Take $v_i\in\Gamma(U_i,TX)$ a local sections where $||v_i||=1$ and $v_i(z)\in E$. Let $L_z:T_z\rightarrow\C^2$ be the unique linear isometry such that $T_z(v_i(z))=(1,0)$ and $\det(L_z)=1$. The map $L_z$ is unique because the only element of the group $SU(2,\C)$ that fix the vector $(1,0)$ is the identity map. Since each $v_i$ is unique and varies continuously, we conclude that $z\mapsto L_z$ is continuous.

\smallskip

Now, take $v_i^*(z)=L_z^{-1}((0,1))\in\Cb_z$. We conclude that  $v^*_i\in\Gamma(U_i,TX)$, $||v^*_i||=1$, and $v^*_i(z)\in E^\perp$. Define the splitting
\[F_z=\bpl v\in T_z\, : \, (v|v_i-v^*_i)_z=0\bpr.\]
It is easy to see that $F_z$ is independent of the choice of $v_i$'s and $F$ is a continuous splitting. We conclude that $F$ define a global section $\tau\in\Gamma(X,\P(X))$, this is, for any $u\in F_z$ we have $\tau(z)=[u]$.

\smallskip

Finally, we define $H_z$ as the unique M$\rm{\ddot{o}}$bius transformation such that\linebreak $H_z(\sigma(z))=0$, $H_z(\sigma^*(z))=\infty$, and $H_z(\tau(z))=1$; more precisely we define $H_z([v])=[L_z(v)]$. This finishes the proof of the claim.

\smallskip

Continuing with the proof of the Proposition, if $(U,\phi)$ is a local of the bundle $\P(X)$, by continuity of the sections, the local expression in $U$ of $H_z$ is a continuous function. More precisely, there exist a continuous family
$\widetilde{H}:U\times\Cb\rightarrow U\times \Cb$ with
\[\widetilde{H}_z(w)=\frac{a_zw+b_z}{c_zw+d_z}\]
where the maps $z\mapsto a_z,\ldots,z\mapsto d_z$ are continuous, $H_z=\widetilde{H}_z\circ\phi$,
$\widetilde{H}_z\circ\widetilde{\sigma}(z)=0$, and
$\widetilde{H}_z\circ\widetilde{\sigma^*}(z)=\infty$, where $\widetilde{\sigma}=\sigma\circ\phi$ and
$\widetilde{\sigma^*}=\sigma^*\circ\phi$. It follows that the function $H=(id,H_*)$ is an homeomorphism and an isometry in each fiber.
\end{proof}

\smallskip

\begin{rem}\label{rem:unit:lift}
After previous proposition we can assume that the bundle $\P(X)$ is in fact the trivial bundle $X\times\Cb$. Moreover, given a section $\sigma\in\Gamma(X,\P(X))$ we can lift this section to the trivial bundle $X\times\C^2$ as a global section $v\in\Gamma(X,X\times\C^2)$ such that $||v||=1$ and if we write $v=(v_1,v_2)$ then
$\sigma(z)=(z,[v_1(z):v_2(z)])$; this helps us to find global expressions of the section in the projective bundle.
\end{rem}

\smallskip

With this remark we can define.

\smallskip

\begin{defi}\label{def:unit:lift}
A section $v$ as in Remark \ref{rem:unit:lift}, is called a {\bf unitary lift} of the section $\sigma$.
\end{defi}

\subsection{Equivalence of Contractive sections}

We denote the unit disc in $\C$ by $\D$. Given $\xi\in\Cb$ an $r>0$, we denote the {\it $\rho$--ball with center at $\xi$ and radius $r$} in the spherical metric by $B(\xi,r)$. For any isometry in the Riemann sphere $L$ with $L(0)=\xi\in\Cb$, the set $L(r\overline{\D})$ does not depend on $L$. We will denote this set by $D_r(\xi)$ and is called {\it disc of radius $r$ centered at $\xi$}. Moreover,  we have that for any $r$, the disc $D_r(\xi)$ is equal to $B(\xi,\varepsilon)$ where $\varepsilon$ satisfies the equation
\[\frac{r}{\sqrt{1+r^2}}=\sin\left(\frac{\varepsilon}{2}
\right).\]
This last equation goes from the relation between the chordal and spherical metric (see for example, \cite{conw}).

\smallskip

\begin{pr}\label{contractive_section}
Let $\sigma\in\Gamma(X,\P(X))$ be a $M$--invariant section. Then the following statement are equivalents:
\begin{itemize}
\item[{\it i.}]The section $\sigma$ is contractive.
\item[{\it ii.}]There exist $0<\eta<1$ and $k>0$ such that $g(k,\sigma(z))<\eta$ for all
$z$ in $X$.
\item[{\it iii.}]There exist $k>0$ and $r>0$ such that
$M^k_z(\overline{D}_r(\sigma(z)))\subset D_r(\sigma(f^k(z)))$.
\item[{\it iv.}]There exist $k>0$ and $R>0$ such that for all $0<r\leq R$, $M^k_z(\overline{D}_r(\sigma(z)))\subset D_r(\sigma(f^k(z)))$.
\end{itemize}
\end{pr}
\begin{proof}
Clearly $(i)$ implies $(ii)$. To see that $(ii)$ implies
$(i)$, define
\[C_j=\sup\bpl g(j,\sigma(z))\,:\,z\in
X\bpr\]
for $j=0,\ldots,k-1$. We conclude that for every $s\geq0$
\[g({sk+j},\sigma(z))=g({sk},\sigma(f^j(z)))g(j,\sigma(z))\leq
C_j\eta^s=C_j\eta^{-j/k}[\eta^{1/k}]^{sk+j}\leq
C\lambda^{sk+j}\]where
$C=\sup\bpl C_j\eta^{-j/k}\,:\,j=0,\ldots,k-1\bpr$ and
$\lambda=\eta^{1/k}<1$.

\smallskip

Also it is clear that $(ii)$ is equivalent with $(iii)$ and that $(iv)$ implies $(ii)$ and $(iii)$.

\smallskip

To prove that $(ii)$ implies $(iv)$, we consider $v=(v_1,v_2)$ the unitary lift of $\sigma$. Consequently, if $\sigma^*$ is the antipodal point of $\sigma$ then $v^*=(\overline{v}_{2},-\overline{v}_{1})$ is an unitary
lift of $\sigma^*$. We consider
\begin{equation}\label{matrix_02}
B_z=
\left(\begin{array}{cc}
       \overline{v}_{2}(z)&v_{1}(z)\\
       -\overline{v}_{1}(z)&v_{2}(z)
     \end{array}\right)
\end{equation}
and denote the M$\rm{\ddot{o}}$bius transformation related with them by $H_z$. It follows that $H_z$ is an isometry of the Riemann sphere.

\smallskip

Now we take $N_z=H^{-1}_{f(z)}\circ M_z\circ H_z$. If we define the cocycles $N$ and $H$ by $N=(f,N_*)$ and $H=(id,H_*)$, then $H\circ N=M\circ H$ and that the null section $\xi_0\equiv0$ is $N$--invariant. We conclude that $N_z$ has the form
\[N_z(\xi)=\frac{\xi}{\beta_z\xi+\alpha_z}.\] Notice that
\[g(1,\xi_0)=g(1,\xi_0,N)=\left|\frac{1}{\alpha_z}\right|\]
so from hypothesis there exist constants $0<\eta<1$ and $k\ge0$ such that $g(k,\xi_0(z))=|\alpha_z^k|^{-1}<\eta$ for every $z$ in $X$. Consequently there exists a $R>0$ such that for every $0<r\leq R$ and $z\in X$, we have $g({k},\xi)\leq \eta$ for every $\xi\in D_r(\xi_0(z))$. It follows that there exist a constant $0<\lambda<1$ and a $C>0$ such that $g({n},\xi)\leq C\lambda^n$ for every $\xi$ in $D_r(\xi_0(z))$ and $n\in\N$. The previous observation implies that $|N^n(\xi)|< C\lambda^nr$ for every $\xi \in D_r(\xi_0)$ and $n\ge1$, as required.
\end{proof}

\smallskip

\begin{cor}\label{local_manifold}
Let $\sigma$ and $R>0$ satisfying the item (iv) of the Proposition \ref{contractive_section}. Then for every $\xi\in D_r(\sigma(z))$ and $0<r\leq R$, we have
\[\lim_{n\rightarrow+\infty}\rho(M^n_z(\xi),M^n_z(\sigma(z)))=0\]
where $\rho$ is the spherical metric.
\end{cor}

\subsection{Expansive/Contractive Direction}

In this subsection we explain some properties of the function $g$. The main goal is establish the uniqueness of direction asymptotically expansive. This fact will be used recurrently in this work.

\smallskip

\begin{lm}\label{formula}If $\xi_i$ with $i=1,2$ are two different directions in $\Cb_z$ and $u_i$ is an unitary vector that generate the direction $\xi_i$ for $i=1,2$, then
\[g(n,\xi_1)g(n,\xi_2)=\left(\frac{\sin(\measuredangle(A^n_zu_1,
A^n_zu_2))}{\sin(\measuredangle(u_1,u_2))}\right)^2,\]
for any $n\in \Z$.
\end{lm}
\begin{proof}
Let $x, y\in\Ctwo$ and denote the area of the polygon formed by the vertices $0$, $x$, $x+y$ and $y$
by $\phi(x,y)$. Then we have the equality \begin{equation}\label{01}
\phi(x,y)=|x|\cdot|y|\cdot\sin(\measuredangle(x,y))=\sqrt{\det([x\, \,
y]^*\cdot[x\, \, y])}=|\det([x\, \, y])|,
\end{equation}
where $[x\, \, y]$ is a column matrix and $[x\, \, y]^*$ denote its
transposed conjugate. Then, it is easy to see that
$\phi(Ax,Ay)=|\det(A)|\phi(x,y)$ for any linear map $A$ in
$\Ctwo$.

\smallskip

Then from equation (\ref{01}) we have that
\[\left(\frac{\sin(\measuredangle(A_z\sp nu_1,A_z\sp
nu_2))}{\sin(\measuredangle(u_1,u_2))}
\right)^2=\frac{\phi(A_z\sp nu_1,A_z\sp nu_2)^2/|A_z\sp
nu_1|^2\cdot|A_z\sp nu_2|^2}{\phi(u_1,
u_2)^2/|u_1|^2\cdot|u_2|^2}
=\frac{|\det(A_z\sp n)|^2}{|A_z\sp nu_1|^2\cdot|A_z\sp nu_2|^2}.\]
According to the equation (\ref{g_n_iterate}) and the previous
equality, it follows that
\[g(n,\xi_1)g(n,\xi_2)=\frac{|\det(A_z\sp n)|^2}{|A_z\sp
nu_1|^2\cdot|A_z\sp
nu_2|^2}=\left(\frac{\sin(\measuredangle(A_z\sp nu_1,A_z\sp
nu_2))}{\sin(\measuredangle(u_1,u_2))}
\right)^2.\]
\end{proof}

\smallskip

\begin{lm}\label{uniq_directions}
An expansive (contractive) direction is unique.
\end{lm}
\begin{proof}
Take $z\in X$ and two different directions that are expanded for the future $\xi_1, \xi_2\in\Cb_z$, that is, there exist $C>0$ and $\lambda>1$ such that $g(n,\xi_i)\ge C\lambda\sp n$ for each $n\ge0$, and $i=1,2$. If $u_i$ is an unitary vector that generate the direction $\xi_i$ for $i=1,2$, we conclude that $\measuredangle(u_1,u_2)>0$.
From the preceding Lemma, we have that
\[C\sp 2\lambda^{2n}\leq
g(n,\xi_1)g(n,\xi_2)=\left(\frac{\sin(\measuredangle(A^n_zu_1,
A^n_zu_2))}{\sin(\measuredangle(u_1,u_2))}\right)^2<\frac{1}{
\sin(\measuredangle(u_1,u_2))}\]
which is a contradiction. For the case that we have expansion for the past, the same proof holds.
\end{proof}

\smallskip

\subsection{Module}
A {\bf double connected domain} in $\Cb$ is a open connected set $U$ such that its complement has two connected component. The definition of the module of a double connected domain is based in the following mapping theorem: {\it Every double connected domain $U$ is biholomorphic to a ring domain of the form
\[A(r_1,r_2)=\bpl z\in\C\, : \,0\leq r_1<|z|< r_2\leq\infty\bpr\]
and is called a canonical image of $U$.}

\smallskip

If $r_1>0$ and $r_2<\infty$ for one canonical image of $U$, then the ratio of the radii $r_2/r_1$ is the same for all canonical images of $U$. Then the number
\[\mod(U)=\log\left(\frac{r_2}{r_1}\right)\]
determines the conformal equivalence class of $U$ and is called the module of $U$. Otherwise, we define $\mod(U)=\infty$ and this happens if and only if at least one boundary component of $U$ consists of a single point.

\smallskip

The following proposition will be crucial in the proof of Theorem \ref{pr:ds_vs_hyp}.

\smallskip

\begin{pr}\label{mod_cor}
Let $\, D_1,\, D_2,\, D_3,\,\ldots$ be conformal discs in $\Cb$ such that for every $i\geq1$ we have $\overline{D}_i\subset D_{i+1}$. If there exist a constant $\kappa>0$ such that $\mod(\textrm{{\rm int}}(D_{i+1}
\setminus \overline{D}_i))\geq\kappa$, then the set
$D=\cup_nD_n$ is biholomorphic to $\C$.
\end{pr}

\smallskip

See \cite{M,LV} for details.

\subsection{Proof of Theorem \ref{pr:ds_vs_hyp}}\label{ss:proof}

We define the {\bf stable set} at the point $\xi\in\Cb_z$  of a cocycle $M$ as the set
\[W^s(\xi)=\left\{ w\in \Cb_z\,: \,
\lim_{n\rightarrow\infty}\rho(M_z^n(w),M_z^n(\xi))=0\right\}\]
and the {\bf local stable set of size} $\varepsilon>0$ by
\[W^s_\varepsilon(\xi)=\left\{ w\in W^s(\xi)\,: \,
\rho(M_z^n(w),M_z^n(\xi))<\varepsilon \,,\,
\, \textrm{for all}\, \, n\in \N\right\}.\]
The unstable set is defined in the same way, for the inverse cocycle $M^{-1}$.

\smallskip

We can write the stable (resp. unstable) set in terms of backward (resp. forward) iterates of the local stable (resp. local unstable) sets. In fact, given $\varepsilon>0$ we have
\[W^s(\xi)=\bigcup_{n=0}^\infty
M^{-n}_z\left(W^s_\varepsilon(M^n(\xi))\right)\, \, \textrm{and}\, \, \,W^u(\xi)=\bigcup_{n=0}^\infty
M^{n}_z\left(W^u_\varepsilon(M^{-n}(\xi))\right).\]

\smallskip

Also we have the following statement.

\smallskip

\begin{lm}\label{lm:s-manifold}
Let $\sigma$ be a contractive section for $M$. Then there exist constants $k\ge0$ and $r>0$, such that
\[W^s(\sigma(z))=\bigcup_{t\geq0}M^{-tk}_z\left(D_r(\sigma(f^{tk}
(z)))\right).\]
\end{lm}
\begin{proof}
From Corollary \ref{local_manifold} it follows that $D_r(\sigma(f^{tk}(z)))$ is the local stable set. Since this is uniformly contractive by the cocycle, we have our Lemma.
\end{proof}

\smallskip

\begin{proof}[{\bf Proof of Theorem \ref{pr:ds_vs_hyp}}]

\noindent
Our proof goes through a series of claims.

\smallskip

\noindent
{\bf Claim 1:} {\it A linear cocycle $A$ has dominated splitting if and only if the cocycle $M([v])=[Av]$ is hyperbolic.}
\begin{proof}[{\bf Proof of Claim 1}] Suppose that $A$ has dominated splitting.

\smallskip

First, let $\{(U_i,\varphi_i):i=1,\ldots,l\}$ be a family of local chart of the bundle $TX$ such that $X=\cup_i U_i$. Let $u_i\in\Gamma(U_i,E)$, and let $v_i\in\Gamma(U_i,F)$ such that $||u_i(z)||_z=||v_i(z)||_z=1$ for every $z\in
U_i$ and $i=1,\ldots,l$. For each $z\in U_i$ define \[\tau(z)=(z,[u_i(z)])\, \, \textrm{ and } \, \, \sigma(z)=(z,[v_i(z)]).\]
It follows that both $\sigma$ and $\tau$ are well--defined continuous global section in $TX$, which are $M$--invariant.

\smallskip

Now, let $K(\beta,F)$ be an $A\sp l$--invariant cone field and denote the set \linebreak$p_I(K(\beta,F_z))$ by $D_z$. Without loss of generality, we can assume that $l=1$ (See Remark \ref{rem:ds}.) We recall that $D_z$ is a closed conformal disc, that is, a biholomorphism image of the closed unitary disc. Note that:
\begin{enumerate}
\item $\sigma(z)\in D_z$,
\item $M_z(D_z)\subset {\rm int}(D_{f(z)})$,
\item $\sigma(z)=\bigcap_{n\geq0}M^{n}_{f^{-n}(z)}(D_{f^{-n}(z)})$.
\end{enumerate}
Item 1 and item 2 it follows directly from definition. Item 3 it follows from equation (\ref{invariant_splitting}). Note also that item 2, implies that $\cap_{n=0}\sp m M^{n}_{f^{-n}(z)}(D_{f^{-n}(z)})=M^{m}_{f^{-m}(z)}(D_{f^{-m}(z)})$. From compactness of $X$ there exists $r>0$ such that $D(\sigma(z),r)\subset D_z$. We conclude that there exists $k>l$ such that
\[M_{f^{-k}(z)}^{k}(D(\sigma({f^{-k}(z)}),r))\subset
\textrm{int}(D(\sigma(z),r)).\]

\smallskip

Finally, since the splitting varies continuously (and consequently the cone fields) and by compactness of $X$, we conclude that the constant $k$ is independent of the choice of the point $z\in X$. By Proposition \ref{contractive_section}, we conclude that $\sigma$ is contractive. An argument similar applied to $\tau$ and $M\sp{-1}$, implies the hyperbolicity of $M$.

\smallskip

Conversely we suppose that $M$ is hyperbolic.

\smallskip

First, we denote by $E_z$ and $F_z$ the sets $p_I^{-1}(\tau(z))\cup\bpl 0\bpr$ and
$p_I^{-1}(\sigma(z))\cup\bpl 0\bpr$ respectively. Clearly this define a $A$-invariant splitting.
Take $u$ and $v$ unitary lifts of $\tau$ and $\sigma$ respectively (See Definition \ref{def:unit:lift}.) Remember that the unitary lift are element of the section of the trivial vector bundle $X\times\C^2$.

\smallskip

Now, we construct a hyperbolic cocycle $N$ conjugated
with $M$. Write $u=(u_1,u_2)$ and $v=(v_1,v_2)$. Define
$H_z$ as the M$\rm{\ddot{o}}$bius transformation related with the matrix
\begin{equation}\label{matrix_01}
B_z=\left(\begin{array}{cc}
       u_{1}(z)&v_{1}(z)\\
       u_{2}(z)&v_{2}(z)
     \end{array}\right),
\end{equation}
define $N_z=H^{-1}_{f(z)}\circ M_z\circ H_z$, and also define the cocycles $N$ and $H$ by $N=(f,N_*)$ and $H=(id,H_*)$. Clearly $H\circ N=M\circ H$. We assert that the cocycle $N$ is hyperbolic. In fact,  note that by construction, the sections $\xi_\infty$ (resp. $\xi_0$)  that associates at each point the point at infinity (resp. the zero point), are $N$--invariant. Consequently, we have that $N_z(\xi)=\lambda_z\xi$. Since $M$ is hyperbolic, then $M^k$ shrinks small closed disc around $\sigma$ and expands small disc around $\tau$ for some $k\geq0$ (see Proposition \ref{contractive_section} ). Thus a similar phenomena holds for $N^k$. We conclude that $|\lambda_z^k|$ is less than one for every $z\in X$. Moreover, the compactness of $X$ allows to take this constant  uniformly in $z$, that is, there exist $0<\eta<1$ such that $|\lambda_z^k|<\eta$. Since $g(k,\xi_0(z))=|\lambda^k_z|<\eta$ and $g(k,\xi_\infty(z))=|\lambda^{-k}_z|>\eta\sp{-1}>1$ it follows the hyperbolicity of $N$.

\smallskip

Finally, let $\bpl (U_i,\varphi_i)\bpr$ be a finite family of local chart of the bundle $TX$ such that $X=\cup_i U_i$. Let  $\widetilde{u}_i, \widetilde{v}_i\in\Gamma(U_i,TX^*)$ such that $\widetilde{u}_i\in E_z$, $\widetilde{v}_i\in F_z$ and both are unitary. Then we have
\[p_I(\widetilde{u}_i)=\tau=(id,[u])\, \, \, \, \textrm{and}\, \, \,
\,
p_I(\widetilde{v}_i)=\sigma=(id,[v]).\]
Take $z\in U_i$ with $f^n(z)\in U_j$ and write
$A^n_z\widetilde{u}_i(z)=k^n_z\widetilde{u}_j(f^n(z))$,
$A^n_z\widetilde{v}_i(z)=l^n_z\widetilde{v}_j(f^n(z))$ and
$\widetilde{A}^n_z=(\varphi_j)^{-1}_z\circ A^n_z\circ(\varphi_i)_z$
then
\[B^{-1}_{f^n(z)}\widetilde{A}^n_zB_z=\left(
\begin{array}{cc}
k^n_z&0\\
0&l^n_z
\end{array}\right)
\]
is the matrix related with the M$\rm{\ddot{o}}$bius transformation $N_z\sp n$. We conclude that $\lambda^n_z=k^n_z/l^n_z$. Now taking
$u^\prime=su_i(z)$ in $E_z$ and
$v^\prime=tv_j(f^n)$ in $F_{f^n(z)}$ with $|s|=|t|=1$, then
\[||A^n_zu^\prime||\cdot||A^{-n}_{f^n(z)}
v^\prime||=||A^n_zu_i(z)||\cdot||A^{-n}_{f^n(z)}
v_j(f^n(z))||=|k^n_z|\cdot|l^n_z|^{-1}=|\lambda^n_z|\leq
C\lambda^n\]for $n\geq1$, so $A$
has dominated splitting.
\end{proof}

\smallskip

\noindent
{\bf Claim 2:} {\it Let $\sigma$ be a contractive section for $M$ (resp. $\tau$ be a expansive section). Then for every $z\in X$, $W^s(\sigma(z))$ (resp. $W^u(\tau(z))$) is biholomorphic to $\C$.}
\begin{proof}[{\bf Proof of Claim 2}]
Take $k$ and $r$ as in Lemma \ref{lm:s-manifold} and define
\[D_t=M^{-tk}_z\left(D_r(\sigma(f^{tk}(z)))\right).\]
It is clear that $D_{t-1}\subsetneq D_{t}$ and the function $M^{tk}_z$ maps biholomorphically $D_t\setminus D_{t-1}$ on
\[A_t=D_r(\sigma(f^{tk}(z)))\setminus
M^{k}_{f^{(t-1)k}(z)}(D_r(\sigma(f^{(t-1)k}(z)))),\]
so the module $\mod(D_t\setminus D_{t-1})$ and $\mod(A_t)$ are equal. We take $0<\eta<1$ uniformly in $X$ such that $g(k,\sigma)\leq\eta$. It follows that $\mod(A_t)\geq\log(1/\eta)$. By Corollary \ref{mod_cor}, the claim is proves.
\end{proof}

\smallskip

\noindent
{\bf Claim 3:} {\it A projective cocycle $M$ is hyperbolic if and only if at least one of the following equivalent conditions hold:
\begin{itemize}
\item[a.] There exist a section that is a contraction.
\item[b.] There exist a section that is an expansion.
\end{itemize}}
\begin{proof}[{\bf Proof of Claim 3}]
We need only show that (b) imply (a), because the reciprocal direction follows using the inverse cocycle.

\smallskip

From the previous Claim, we have that $\Cb_z\setminus
W^s(\sigma(z))=\bpl \tau(z)\bpr$. Since $W^s$ varies continuously and is $M$--invariant, it follows that $\tau$ is a $M$--invariant section. By the definition of $\tau$, it follows that small disc around of $\tau$ are contracted
uniformly by $M^{-1}$, and it follows that $\tau$ is an expansion.
\end{proof}

\smallskip

\noindent
{\bf Claim 4:} {\it If a projective cocycle $M$ satisfy item 4 of this Theorem, then there exist a section that is a expansion.}
\begin{proof}[{\bf Proof of Claim 4}]First, we denote by $u_z$ some unitary vector that define the direction $\tau_z$. Since we have uniqueness of an expansive direction (Lemma \ref{uniq_directions}), we conclude that $M(\tau_{f^{-1}(z)})=\tau_z$.

\smallskip

Now, let $z_n\rightarrow z$ in $X$, then
$\tau_{z_n}\rightarrow \tau_z$. In fact, by compactness there exists some adherence point for the sequence $(\tau_{z_n})_n$, named $\tau'\in\Cb_z$ that is expansive for the future. From Lemma \ref{uniq_directions}, it follows that $\tau'$ is equal to $\tau_z$, and hence we have that the function $z\mapsto\tau_z$ is continuous.

\smallskip

Finally, from the hypothesis we have that $\tau$ is an expansion. Then, from Claim 3 and Claim 1 we have that $X$ has dominated splitting.
\end{proof}
This conclude the proof of Theorem \ref{pr:ds_vs_hyp}.
\end{proof}

\smallskip

\section{Critical Points and Main Theorem}
To give a precise definition of critical point and state out Theorem A, firstly we introduce some technical notations. Let $\Delta$ be the set defined by
\[\Delta=\{\beta=(\beta_-,\beta_+)\, : \, 0<\beta_+\le\beta_-<1\}.\]
In $\Delta$ we have a partial order. In fact, let $\alpha,\beta\in\Delta$, we say that $\beta\ge\alpha$ if and only if $\beta_\pm\sp{\pm1}\ge\alpha_\pm\sp{\pm1}$.

\smallskip

\begin{defi}
\begin{enumerate}
\item Let $\beta\in\Delta$ and $n_-\le0\le n_+$ integers. We say that $x\in X$ is a {\bf $\beta$--critical point at the times $(n_-,n_+)$} (for the linear cocycle $A$), if there exist a direction $\xi_x\in\Cb_x$ such that for every $n\ge0$ we have that $g(\pm n,M\sp{{n}_\pm}\xi_x)\ge\beta_{\pm}\sp{\pm n}$. The direction $\xi_x$ will be called {\bf critical direction}.

\item We say that $x$ is a {\bf $\beta$--critical point} if this is a $\beta$--critical point at the times $(0,0)$. We denote the set of all $\beta$--critical point by $\crit{\beta}$.

\item We say that $y$ is a {\bf $\beta$--critical value} if $y$ is a $\beta$--critical point for the linear cocycle $A\sp{-1}$. We denote the set of all $\beta$--critical value by $\cv{\beta}$.
\end{enumerate}
\end{defi}

\begin{rem}\label{crit_inq}
It follows easily from the previous definition that if $\alpha,\beta\in\Delta$ such that $\beta\ge\alpha$ then $\crit\beta\subset\crit\alpha$.
\end{rem}

\smallskip

Note that from Lemma \ref{uniq_directions}, the critical direction is unique. In Section \ref{sec:6} we explain in details a series of properties of critical points. For the moment, we have the following property.

\smallskip

\begin{rem}The previous definition say that \[\crit{\beta}=\{x\in X\,:\,\exists\, \xi_x\, \textrm{ such that }\, g(\pm n,\xi_x)\ge\beta_{\pm}\sp{\pm n}\, , \, \, n\ge0\}\]
and
\[\cv{\beta}=\{x\in X\,:\,\exists\, \xi_x\, \textrm{ such that }\, g(\mp n,\xi_x)\ge\beta_{\pm}\sp{\mp n}\, , \, \, n\ge0\}\]
\end{rem}

\smallskip

In order to state the Main Theorem, we need the following.

\smallskip

\begin{defi}
Given $0<b<1$, we say that $X$ is {\bf $b$--asymptotically dissipative} (for a cocycle $A$) if there exists a positive constant $C>0$ such that for every $z\in X$, $|\det(A^n_z)|\leq Cb^n$ for every $n\geq0$.
\end{defi}

\smallskip

We recall that: given a cocycle $A=(f,A_*)$, we say that $f$ has no attractors, if all measure $f$--invariant is partially hyperbolic (See Definition \ref{def:partially_hyp_measure}). Our Main Theorem is the following.

\smallskip

\begin{teoa}\label{main_thm} Let $A=(f,A_*)$ be a linear cocycle over $X$ such that $f$ has no attractors. Assume that $X$ is $b$--asymptotically dissipative for the cocycle $A$. Then $A$ has dominated splitting if and only if for each $\beta\in\Delta$ which $\beta_+>b$, the set $\crit\beta$ is an empty set.
\end{teoa}

\smallskip

In the remainder of this section, we state several tools necessaries to prove the Main Theorem.

\smallskip

\subsection{Blocks of Domination}

\begin{defi}
Given $\alpha\in(0,1)$ and $\gamma>0$, we define the {\ the blocs of domination} as the sets
\[\gamma H^\pm(\alpha)=\bpl z\in X:\exists\, \, \xi\in\Cb_z\textit{\,
such that\, }g(\pm n,\xi)\geq \gamma\alpha^{-n},\, \forall\, n\ge
0\bpr,\]
and the sets
\[\gamma\Hc^\pm(\alpha)=\bpl z\in X:\exists\, \, \xi\in\Cb_z
\textit{\, such that\, }g(\pm n,\xi)>\gamma\alpha^{-n},\,
\forall\, n\ge 0\bpr.\]
\end{defi}

\smallskip

\begin{rem}
\begin{enumerate}
\item When $\gamma=1$, we denote the set $\gamma H\sp\pm(\alpha)$ (resp. $\gamma \Hc\sp\pm(\alpha)$) by $H\sp\pm(\alpha)$ (resp. $\Hc\sp\pm(\alpha))$.
\item It is easy to see that if $0<\alpha\le\alpha\sp\prime$ then $ \gamma H\sp\pm(\alpha)\subseteq \gamma H\sp\pm(\alpha')$ and
$\gamma\Hc\sp\pm(\alpha)\subset\gamma\Hc\sp\pm(\alpha')$.
\end{enumerate}
\end{rem}

\smallskip

The next Theorem establishes condition for the existence of blocks of domination. We prove them in subsection \ref{proofteoblock}.

\smallskip

\begin{teo}\label{thm_A} 
Let $A=(f,A_*)$ be a linear cocycle over $X$ such that $f$ has no attractors. Then for any $1>\beta>b$ and $1\le\gamma<b\sp{-1}\beta$ the blocks of domination $\gamma H^+(\beta)$ and $\gamma\sp{-1}H\sp-(\beta)$ are not empty
compact sets. Moreover, the sets
\[X_0^+=\cup_{n\in\Z}f^n(\gamma H^+(\beta))\, \, \, \textrm{and}\,
\, \, X_0^-=\cup_{n\in\Z}f^n(\gamma\sp{-1}H\sp-(\beta))\]
have total measure for any invariant measure $\nu$ with support in $X$.
\end{teo}

\smallskip

In \cite{P-RH}, the authors introduce the notion of critical point in terms of the block of domination. We decided introduce this notion independently of them. In subsection \ref{subs:block}, we explain in details this relations.

\subsection{Pliss's Lemma}

The following Lemma is a remarkable result known as Pliss lemma, which is frequently used in this paper.

\smallskip

\begin{lm}{\rm{\bf (Pliss's Lemma)}}
Given $0<\gamma_1<\gamma_0$ and $a>0$, there exist
$N_0=N_0(\gamma_0,\gamma_1,a)$ and
$\delta_0=\delta_0(\gamma_0,\gamma_1,a)>0$ such that for any sequences of numbers $(a_l)_{l=0}^{n-1}$ with $n>N_0$, $a^{-1}<a_l<a$ and $\prod^{n-1}_{l=0}a_l\geq\gamma_0^n$ we have that:

\noindent
if
\begin{equation}\label{pliss_set}
Ht=\left\{ 0\leq k<n\, : \,\forall\, k<s<n,\, \textrm{we have that}\,
\prod_{l=k+1}^{s}a_l\geq\gamma_1^{s-k} \right\},
\end{equation}
then $\#Ht\geq n\cdot\delta_0$.
\end{lm}

\smallskip

\begin{rem}
When $k\in Ht$, $k$ is called a {\bf hyperbolic time}.
\end{rem}

\smallskip

As a corollary, we have the following result.

\smallskip

\begin{cor}\label{ap_p_lemma}
Given $0<\gamma_1<\gamma_0$ there exist $N_0$
and $\delta_0$ positive constants such that:
If for $z\in X$ there exists $\xi\in\Cb_z$ satisfying
$g(n,\xi)\geq\gamma_0^n$ (resp. $g(-n,\xi)\geq\gamma_0^n$) for some
$n\geq N_0$, then there exists $0\leq j<n$ such that
$n-j>n\delta_0-1$ and
\[g(i,M^j(\xi))\geq\gamma_1^i\, \, \, \, \textrm{for every}\, \, \,
\, 0<i\leq n-j,\]
(resp. $g(-i,M^{-j}(\xi))\geq\gamma_1^i$\, \,  for every\, \,
$0<i\leq n-j$).
\end{cor}
\begin{proof}
From Pliss's Lemma, taking $a_l=g(1,M^l(\xi))$ for $l=0,\ldots,n-1$ then
$g(n,\xi)=\prod^{n-1}_{l=0}g(1,M^l(\xi))\geq\gamma_0^n$. Let $k_0$ be the lowest hyperbolic time. We have that $n-k_0\geq n\delta_0$, and for every $k_0<s<n$
\[\gamma_1^{s-k_0}\leq\prod^{s}_{l=k_0+1}g(1,M^l(\xi))=g(s-k_0,M^{k_0+1}(\xi)).\]
Hence it is enough to take $j=k_0+1$, and we have the corollary.
\end{proof}

\subsection{Criteria for Domination}

Now we present a criteria for the existence of dominated splitting that is essential in the proof of our Main Theorem.

\smallskip

\begin{teo}{\rm{\bf (Criteria for Domination)}} Let $X$ be a $b$--asymptotically dissipative for a cocycle $A=(f,A_*)$ such that $f$ has no attractor, and let \linebreak$1>\beta>b$.

\smallskip

If there exist $k_0,m_0>0$ such that for all $z\in X$, there exists one direction $\xi_z\in\Cb_z$ satisfying that
\begin{equation}\label{eq_cd_I}
g(k_0,M^m(\xi_z))\leq\beta\sp{-k_0},\, \, \, \, \textrm{for every}\, \, \, \,m>m_0;
\end{equation}
then $X$ has dominated splitting.
\end{teo}

\smallskip

To prove the Criteria for Domination we need of the following.

\smallskip

\begin{pr}\label{crit_dom_2}
Suppose that there exist $k_1,m_1>0$ and $\gamma<1$ such that for any $z\in X$, there exists one direction $\xi_z\in\Cb_z$ such that
\begin{equation}\label{eq:04}
g({k_1},M^m(\xi_z))<\gamma,\, \, \, \, \textrm{for every}\, \, \, \,m>m_1;
\end{equation}
then $X$ has dominated splitting.
\end{pr}
\begin{proof} Fix $z_0\in X$ and denote by $\xi_0=M^{m_1}(\xi_{z_0})$
and $\xi_t=M^t(\xi_0)$, then we have that $g(k_1,\xi_t)<\gamma$ for every $t\geq0$. Let us take, for $j=0,\ldots,k_1-1$\[C_j=\sup\bpl
g(j,w)\, :\, w\in\Cb_z\,, \, \, z\in X\bpr,\]
it follows that
\[g(nk_1+j,\xi_0)=g(j,\xi_0)g(nk_1,\xi_j)\leq C_j\gamma^n\leq
C\lambda_0^{nk_1+j},\]
where $\lambda_0=\gamma^{1/k_1}<1$ and $C_0=\sup\bpl C_j\gamma^{-j/k_1}\,
: \, j=1,\ldots,k_1-1\bpr$.

\smallskip

To end, for every $z\in X$ let us take $z_0=f^{-m_1}(z)$ and $\sigma_z=M^{m_1}(\xi_{z_0})$, it follows that $g({-n},\sigma_z)\geq
C\lambda^n$, where $C=C_0^{-1}$ and $\lambda=\lambda_0\sp{-1}$. From Theorem \ref{pr:ds_vs_hyp} item \ref{basical_domination}, we conclude the domination.
\end{proof}

\smallskip

A fundamental tool to prove the Criteria for Domination, is the following lemma. This establish that if there exists one direction that is neither contracted nor expanded for the future, then the largest Lyapunov exponent in the omega limit of this point is negative. Our version is a stronger version of Main Lemma in \cite{P-RH}, but enough to
conclude what we want. 
\smallskip

\begin{lm}\label{cnexp}{\rm{\bf (Criteria of Negative Exponent)}} Let $1>\beta_2,\beta_1>b$, let $x\in X$ and let $\xi_x$ be a direction in $\Cb_x$. Suppose that there exist constants $n_0, m_0\in\N$ such that
\begin{itemize}
\item[{\it i)}] $\omega(x)$ is $b$--asymptotically dissipative,
\item[{\it ii)}] $\beta_2^n\leq g(n,\xi_x)$ for every $n\geq n_0$.
\item[{\it iii)}] $g(n,M^{m}\xi_x)\leq\beta_1^{-n}$ for every $m>m_0$
and $n\geq n_0$.
\end{itemize}
Then $\omega(x)$ supports a measure $\mu$ which large exponent is negative, and so, $\nu$ is not a partially hyperbolic measure.
\end{lm}
\begin{proof}
We may assume that $\omega(x)$ only support partially hyperbolic measures, and then the large exponent is not negative. We take $n_k\rightarrow\infty$ such that the limit
\[\mu=\lim_{k\rightarrow\infty}\frac{1}{n_k}\sum_{i=0}^{n_k-1}\delta_{
M^i(\xi_x)}\]
there exists. We denote the support of the measure $\mu$ by $\widehat{K}$. Then $\widehat{K}$ is a compact set of $TX$ and his projection $K=pr(\widehat{K})\subset\omega(x)$ is the support (in $X$) of the measure
\[\mu'=\lim_{k\rightarrow\infty}\frac{1}{n_k}\sum_{i=0}^{n_k-1}\delta_
{f^i(x)}\]
that is the projection of $\mu$ in the first coordinate.

\smallskip

Since $\mu'$ is a $f$--invariant measure, we have that for any $z_0\in\mathcal{R}(A,\mu')$ and $w\in\Cb_{z_0}$, the limit
\begin{equation}\label{eq:02}
\begin{aligned}
\lim_{k\rightarrow\infty}\frac{1}{n_k}\log
\left(\phantom{\overline{A}}\hspace{-8pt}g({n_k},w)\right)&=\lim_{
k\rightarrow\infty}\frac{1}{n_k}\left(\log(|\det
A^{n_k}_{z_0}|)-2\log(|A^{n_k}_{z_0}
w|)\right)\\
&=\lambda^+(z_0)+\lambda^-(z_0)-2\lambda(w)
\end{aligned}
\end{equation}
there exists, where $\lambda(w)$ is the Lyapunov exponent associated with the direction $w$. We denote the limit given in the equation (\ref{eq:02}) by $I(z_0,w)$. Since $\lambda(w)=\lambda\sp\pm(z_0)$, we conclude that $I(z_0,w)$ take only the values $\lambda^+(z_0)-\lambda^-(z_0)$ or $\lambda^-(z_0)-\lambda^+(z_0)$.

\smallskip

Since that $z_0\in\omega(x)$, there exists a sequences $(m_k)_k$ such that $f^{m_k}(x)\rightarrow z_0$. By passing to subsequence if necessary, there exists $w_0\in\Cb_z$ such that $M^{m_k}(\xi_x)\rightarrow w_0$. We conclude that $(z_0,w_0)$ is a point of $\widehat{K}$.

\smallskip

By item $(iii)$ of this Lemma, we have that
\begin{equation}\label{eq:03}
I(z_0,w_0)\leq-\log(\beta_1).
\end{equation}
Moreover, this inequality is true for every $(z,w)\in\widehat{K}$, with $z\in\mathcal{R}(A,\mu')$.

\smallskip

On the other hand, we remark that $\lambda^-(z_0)\leq\log
b<0\leq\lambda^+(z_0)$ and $\lambda^-(z_0)-\lambda^+(z_0)\leq
\lambda^-(z_0)+\lambda^+(z_0)\le\log b$. Hence either
$I(z_0,w_0)\leq\log b$ or $I(z_0,w_0)\ge-\log b$. If we suppose that the second inequality holds, then $I(z_0,w_0)\ge-\log b>-\log\beta_1$ that contradict the equation (\ref{eq:03}). We conclude that for every
$(z,w)\in\widehat{K}$ with $z$ a regular point in the Oseledets sense, the limit $I(z,w)$ is equal to
\[\lim_{k\rightarrow\infty}\frac{1}{n_k}\log(g({n_k},w))=
\lambda^-(z)-\lambda^+(z)\leq\log(b).\]
\noindent
{\bf Claim.} $\mu(\widehat{K}\cap pr^{-1}(\mathcal{R}(A,\mu'))=1$.
\newline
{\bf Proof of the Claim.} The Ergodic Decomposition Theorem assert that: There exists a set $\Sigma$ of full probability in $\P(X)$ such that for all $(z,w)\in \Sigma$ the limit
\[\lim_{n\rightarrow\infty}\frac{1}{n}\sum_{j=0}^{n-1}\delta_{
M^j(z,w)}=\mu_{(z,w)}\]
is an ergodic measure, and that for all $h\in\mathcal{L}^1(\P(X),\mu)$ we have
\[\int \left( \int hd\mu_{(z,w)} \right)d\mu=\int hd\mu.\]
In particular, the projection in the first coordinate
\[\mu'_z=pr\circ\mu_{(z,w)}=\lim_{n\rightarrow\infty}\frac{1}{n}\sum_{
j=0}^{n-1}\delta_{f^j(z)}\]
is ergodic and
\[\int \left( \int \widetilde{h}d\mu'_{z} \right)d\mu=\int \widetilde{h}d\mu'\]
where $\widetilde{h}\in\mathcal{L}^1(X,\mu')$.

\smallskip

On the other hand, note that our claim is true for ergodic measures. Recall also that, since the measures $\mu_z'$ is ergodic, the Lyapunov exponents are invariant functions which are constant in the support of the measures $\mu_z'$. Now, the set $R(z)=pr^{-1}(\mathcal{R}(A,\mu'_z))$ is invariant by the projective cocycle, so has $\mu_{(z,w)}$--measure 0 or 1, but
\[\mu_{(z,w)}(\log(g))=\mu_{(y,t)}(\log(g))=
\lambda^-(y)-\lambda^+(y)\]
for $(y,t)\in R(z)$  $\mu_{(z,w)}$--a.e.. The preceding implies that $\mu_{(z,w)}(R(z))\neq0$, and so is equal to 1, which completes the proof of the claim.

\smallskip

Continuing with the proof of the Lemma, applying Birkhoff Ergodic Theorem to the function $\phi=\log(g)$, $\mu$--a.e. $(z,w)\in\widehat{K}$, there exist the limit
\[\widetilde{\phi}(z,w)=\lim_{n\rightarrow\infty}\frac{1}{n}\sum_{j=0}
^{n-1}\phi\circ M^j(w)=\lambda^-(z)-\lambda^+(z),\]
hence it follows that
\[\log(b)\geq\int\widetilde{\phi}d\mu(z,w)=\int\phi
d\mu(z,w)=\lim_{k\rightarrow\infty}\frac{1}{n_k}\log(g(n_k,\xi_x))\geq\log\beta_2,\]
which is a contradiction, because $\beta_2>b$.
\end{proof}

\smallskip

\begin{proof}[{\bf Proof of Criteria for Domination}]
Let $\beta_0$ be a constant such that $1>\beta_0>\beta>b$. From Proposition \ref{crit_dom_2}, to obtain domination we need only to prove that there exist po\-si\-ti\-ve integers $k_1$ and $m_1$, such that for every $z$, there exists one direction $\xi_z$ satisfying
$g(k_1,M^m\xi_z)<\beta_0^{k_1}$ for all $m>m_1$, as in equation (\ref{eq:04}).

\smallskip

If not, we have that every pair $k_1,\, m_1\in\N$, there exists $z\in X$ such that for every $\xi\in\Cb_{z}$ we have that $g(k_1,M^m\xi)\geq\beta_0^k$ for some $m>m_0$. We conclude, in particular, that for every $k$ there exist $z_k\in X$ and $m_k>k$ such that $g(k,M^{m_k}\xi_{k})\geq\beta_0^k$, where $\xi_{k}$ satisfies the equation (\ref{eq_cd_I}) in the statement of this Proposition.

\smallskip

We take $1>\beta_0>\beta_2>\beta$. Applying Corollary \ref{ap_p_lemma} to the constants $\beta_0$ and $\beta_2$, we conclude that there exists a sequence $(r_k)_{k\geq1}$ with $k-r_k\rightarrow\infty$ such that
\[g(s,M^{r_k}\xi_{k})\geq\beta_2^s,\, \, \, \textrm{for
every}\, \, \, 0<s\leq k-r_k.\]

\smallskip

Taking $z$ and $\varpi$ as an accumulations point of
$(f^{r_k}(z_k))_k$ and $(M^{r_k}(\xi_{k}))_k$, respectively, it follows that
\[\beta_2^n\leq g(n,\varpi),\, \, \, \textrm{for every}\, \, \, n>0.\]

\smallskip

On the other hand, since for every $k$ and  $m\ge m_0$ we have that $g(k_0,M^m(\xi_k))\leq\beta\sp{-k_0}$, so we conclude that there exists a constant $C>0$ such that for every $k>0$, $n\geq0$ and $m\geq m_0$ we have $g(n,M^m(\xi_k))\leq C\beta^{-n}$. Taking $\beta>\beta_1>b$ and $n_0$ large enough we obtain that for $n\ge n_0$, we have that $g(n,M^m(\xi_k))\leq \beta_1^{-n}$.
Passing to the limit we conclude that for all $m\ge m_0$ and $n\geq n_0$
\[g(n,M^m(\varpi))\leq\beta_1^{-n}.\]
In this point, we are in the hypothesis of the Criteria of Negative Exponent, hence there exists an invariant measure that is not partially hyperbolic supported in $X$, which is a contradiction.
\end{proof}

\section{Proof of Main Theorem}

This section is based in the ideas of Sylvain Crovisier (see \cite{C}), for the proof of the same result in the context of $C^2$ generic dimorphism in compact manifolds (vide, \cite{P-RH}). Our exhibition presents significant changes compared with that of Silvan, among others, we
have a different definition of critical point than \cite{P-RH}. Now we present a notion that allows to prove the Main Theorem.

\smallskip

\begin{defi}
Given $1>\beta_0>0$, we say that a projective cocycle $M$ satisfies the property $P(\beta_0)$ if there exist $k_0>0$, such that for every $k>k_0$ there exit $x_k\in X$,
$\xi_k\in\Cb_{x_k}$ and $m_k\geq0$ such that:
\begin{itemize}
\item[{\rm 1.}] $g({-n},\xi_k)\geq \beta_0^{-n}$, for every $1\leq n \leq k$,
\item[{\rm 2.}] $g(k,M^{m_k}(\xi_k))\geq1$.
\end{itemize}
\end{defi}

\smallskip

The next Lemma, will be prove in the next Section.

\smallskip

\begin{lm}\label{num_sub_lemma}
Let $(a_n)_{n\in\Z}\subset \R$ be a sequence. Assume that there exist $n_0\le0\le n_1$ and $-\infty<\delta_+\le\delta_-<0$ such that:
\begin{itemize}
\item[{\it a)}] $a_n-a_{n_0}\ge(n-n_0)\delta_-$, for all $n\le n_0$,
\item[{\it b)}] $a_n-a_{n_1}\ge(n-n_1)\delta_+$, for all $n\ge n_1$.
\end{itemize}
Then there exists ${N}\in[n_0,n_1]$ such that $a_{\pm n+N}-a_N\ge \pm n\delta_\pm$ for all $n\ge0$.
\end{lm}

\smallskip

\begin{pr}\label{pad_imply_adcrit} Let $1>\beta_0>0$. If the projective cocycle $M$ satisfies the property $P(\beta_0)$, then for every $\beta\in\Delta$ which
$(\beta_0,\beta_0)\ge \beta$ the set $\crit{\beta}$ is not empty.
\end{pr}
\begin{proof}
First, let $(\beta_0,\beta_0)\ge \beta=(\beta_-,\beta_+)$. We recall that $\beta_+\le\beta_0\le\beta_-$, and that for each $0\le n\le k$ we have that $g(-n,\xi_k)\ge\beta_0\sp{-n}\ge\beta_-\sp{-n}$.

Now, we will apply the Corollary \ref{ap_p_lemma}. Let $k>0$, $\gamma_0=1$, $\gamma_1=\beta_+$ and let $n_0$ and $\delta_0>0$ be the numbers given by this Corollary.
If we choose $s>n_0$ such that $s\delta_0-1>k$, since that $g(s,M\sp{m_s}\xi_s)\geq1$, then there exist $0\leq j<s$ such that $s-j>s\delta_0-1>k$ and
\[g(i,M\sp{m_s+j}\xi_s)\geq\beta_+\sp
i,\, \, \, \, \textrm{ for every }\, \, \, 0<i\leq s-j.\]
Therefore taking $y_k=x_s$, $\upsilon_k=\xi_s$ and $n_k=m_s+j$, we obtain that for every $k>0$, there exist $y_k\in X$, $\upsilon_k\in\Cb_{y_k}$  and $n_k\geq0$ such that, for all $0<n\leq k$,
\[g(-n,\upsilon_k)\geq \beta_-^{-n}\, \textrm{ and }\,
g(n,M^{n_k}\upsilon_k)\geq\beta_+^{n}.\]

Define $n_0=0$, $n_1=n_k$, $\delta_\pm=\log(\beta_\pm)$ and
\[
a_n=\begin{cases}
n\delta_- & , \, n<-k\\
\log(g(n,\upsilon_k)) & , \, -k\le n\le n_k+k\\
(n-n_k)\delta_++a_{n_k} & , \, n>n_k
\end{cases}.
\]

It is not difficult to see that we are in the hypothesis of Lemma \ref{num_sub_lemma}, and hence there exists $-l_k\in[0,n_k]$ such that $a_{\pm n-l_k}-a_{-l_k}\ge\pm n\delta_\pm$ for all $n\ge0$. From the construction of sequence $(a_n)_n$ we can conclude that $g(\pm n,M\sp{l_k}\xi_x)\ge \beta_\pm\sp{\pm n}$ for all $0\le n\le k$.

Finally, if we take $z_k=f^{l_k}(y_k)$ and
$\omega_k=M^{l_k}(\upsilon_k)$, we have that for each $0<n\leq k$
\[g({-n},\omega_k)\geq\beta_-^{-n}\, \, \, \textrm{ and
}\, \, \, g({n},\omega_k)\geq\beta_+^n.\] To end, take $(z,\omega)$ an adherence point of $(z_k,\omega_k)$, and we have that for $n\geq0$
\[g({-n},\omega)\geq\beta_-^{-n}\, \, \, \textrm{ and
}\, \, \, g({n},\omega)\geq\beta_-^n,\]
then $\crit{\beta}$ is nonempty as asserted.
\end{proof}

\smallskip

We denote by $\supp(X)$ the closed subset of $X$ that support all measure $f$--invariant, i.e.,
\[\supp(X)=\overline{\cup\bpl \supp(\nu)\,:\,\nu\, \textrm{ is }\,
f\textrm{-invariant }\, \bpr}.\]

\smallskip

\begin{lm}If  $\,1>\beta_0>b$, then
$\supp(X)\subset \omega(H^-(\beta_0))$.
\end{lm}
\begin{proof}
Any point in the support of an invariant measure $\nu$ is approximated by regular points. By the proof of Theorem \ref{thm_A}\footnote{See subsection \ref{proofteoblock}}, any $x\in\mathcal{R}(A,\nu)$ has infinitely many iterates in $H^-(\beta_0)$. The previous remark and the Poincar\'e recurrence theorem implies that
\[\supp(\nu)\subset
\omega(H^-(\beta_0))=\overline{\bigcup_{z\in
H^-(\beta_0)}\omega(z)},\]
and this implies that $\supp(X)\subset
\omega(H^-(\beta_0))$.
\end{proof}

\smallskip

\begin{lm}\label{ifnotcd_Xstar_ds}
If there exist $1>\beta_0>b$ such that the property
$P(\beta_0)$ is not satisfied, then the set $\supp(X)$ has
dominated splitting.
\end{lm}
\begin{proof}
As the property $P(\beta_0)$ is not satisfied, then there exists $k>0$ such that for every $x\in X$ and $\upsilon\in\Cb_x$ both $g({-n},\upsilon)<\beta_0\sp{-n}$ for $1\leq n\leq k$, or $g(k,M\sp m\upsilon)<1$ for every $m\geq1$. In particular, for points $x\in H\sp-(\beta_0)$ with critical direction $\xi$, the first inequality can not holds, then
\begin{equation}\label{pistoru}
g(k,M\sp m\xi)<1.
\end{equation}

\smallskip

We claim that $\omega(H^-(\beta_0))$ has dominated
splitting. From equation (\ref{pistoru}), we have that
for every $x\in H\sp-(\beta_0)$ (and denoting his critical direction by $\xi$), $g(k,M\sp m\xi)<\beta_0\sp{-k}$. Let $z\in\omega(x)$ and $(m_l)_l$ be a sequence of positives integers goes to infinity such that $f\sp{m_l}(x)\rightarrow z$. By passing to a subsequence if necessary, there exists a direction $\xi_z\in\Cb_z$ such that  $M\sp{m_l}(\xi)\rightarrow \xi_z$. From continuity of $g$ we conclude that: for every $x\in H\sp-(\beta_0)$, and $z\in\omega(x)$ there exist $\xi_z\in\Cb_z$ satisfying
\[\beta_0\sp{-k}>g(k,M\sp{m+m_l}\xi)=
g(k,M\sp{m}(M\sp{m_l}\xi))\rightarrow g\sp
k(M\sp m(\xi_z)).\]
Hence this implies that $\omega(H^-(\beta_0))$ satisfies the hypothesis of the Criteria for Domination. In particular $\supp(X)$ has dominated splitting.
\end{proof}

\smallskip

We can rewrite the item \ref{basical_domination} of Theorem \ref{pr:ds_vs_hyp} in terms of blocs of dominations. Moreover, we can state.

\smallskip

\begin{lm}\label{dom_and_constant}
Let $\Lambda\subseteq X$ be a compact $f$--invariant. Then the linear cocycle $A$ has dominated splitting in $\Lambda$ if and only if there exist $\gamma,\, {\beta}\sp*>0$ such that $\Lambda\subset\gamma \Hc\sp-(\beta)$, and if and only if $\Lambda\subset\gamma \Hc\sp+(\beta)$,
where $\beta\in({\beta}\sp*,1)$ is arbitrarily.
\end{lm}
\begin{proof}
Let $T\Lambda=E\oplus F$ be a dominated splitting. By Proposition \ref{pr:ds_vs_hyp},  we have that there exist constants $C>0$ and $0<\lambda<1$ such that $g^{-n}(F_z)\geq C\sp{-1}\lambda^{-n}$ and $g^{n}(E_z)\geq C\sp{-1}\lambda^{-n}$. We may take $\lambda$ minimal with this property. Then taking $C\sp{-1}=\gamma$ and $\beta\sp*=\lambda$ we obtain the necessary.

\smallskip

The suffices direction is immediate from item \ref{basical_domination} of Theorem \ref{pr:ds_vs_hyp}.
\end{proof}

\smallskip

\begin{pr}\label{Xnotds_thencd}
If $\,\supp(X)$ has dominated splitting but $X$ does not have dominated splitting, then there exists $\beta\sp*$ (minimal) such that the property $P(\beta_0)$ is satisfied on $X$, for every $\beta_0\in(\beta\sp*,1)$.
\end{pr}
\begin{proof}
Let $\beta\sp*$ as in Lemma \ref{dom_and_constant}. Since that $X$ does not have a dominated splitting and denying the Criteria for Domination, follows that for every positive integer $k$, there exists a point $x_k\in X$, and an integer $m>0$ such that for every direction $\omega\in\Cb_{x_k}$ we have that
\begin{equation}\label{first_eq}
g(k,M^{m}\omega)\geq1.
\end{equation}

\smallskip

On the other hand, the set $\alpha$-limit of $x_k$, that we denote by $\alpha(x_k)$, supports an $f$--invariant measure, hence there exists $z_0\in\alpha(x_k)\cap\supp(X)$. From the preceding Corollary, there exists one direction $\xi_0\in\Cb_{z_0}$ such that
\[g({-n},\xi_0)>\gamma(\beta')\sp{-n},\,
\, \, \textrm{ for every }\, \, \, n\geq1,\]where
$\beta'\in(\beta\sp*,1)$.

\smallskip

Take $\beta\sp*<\beta'<\beta_0<1$ and $k$ fixed.

\smallskip

Let $(n_t)_t\nearrow\infty$ such that $f^{-n_t}(x_k)\rightarrow z_0$. For every positive integer $s$, we can find some neighborhood $U_s\subset\P(X)$ of $\xi_0$ such that for every $\xi\in U_s$, holds that $g({-n},\xi)\geq\gamma(\beta')\sp{-n}$ for all
$1\leq n\leq s$. If we take $t$ great enough, $f^{-n_t}(x_k)$ is inside of the projection in $X$ of neighborhood $U_s$. Hence, there
exists $\xi_s\in\Cb_{f^{-n_t}(x_k)}$ such that
$g({-n},\xi_s)\geq\gamma(\beta')\sp{-n}$ for all $1\leq n\leq s$. Note that for $s$ great enough we have that $g({-s},\xi_s)\geq\beta_0\sp{-s}$, hence we are in the hypothesis of Corollary \ref{ap_p_lemma}.

\smallskip

We conclude that, we can find $s$ and $l_s$ such that $s-l_s>k$ and
\[g({-n},M\sp{-l_s}\xi_s)\geq\beta_0^{-n},\, \, \, \textrm{ for every }\, \, \, 0<n\leq s-l_s,\]
in particular, writing $\upsilon_k=M\sp{-l_s}(\xi_s)$, we conclude that
\[g({-n},\upsilon_k)\geq\beta_0^{-n},\, \, \,
\textrm{ for every }\, \, \, 0<n\leq k.\]

\smallskip

From equation (\ref{first_eq}), we have that there exist $m_k$ such that $g(k,M\sp{m_k}\upsilon_k)\geq1$, so the property $P(\beta_0)$ is satisfied.
\end{proof}

\smallskip

With this in mind, we can prove one direction of Main Theorem. In the next subsection, we present their proof in the other direction.

\smallskip

\begin{proof}[{\bf Proof of Main Theorem:}] {\it If $X$ does not have dominated splitting, then for each 
$\beta\in\Delta\sp+$ with $\beta_+>b$ we have $\crit\beta\neq\emptyset$.}

\bigskip

\noindent
We claim that there exist $1>\beta_0>b$ such that the property $P(\beta_0)$ holds.

\smallskip

In fact, if we assume that for each $1>\beta_1>b$ not holds, then from Lemma \ref{ifnotcd_Xstar_ds} we conclude that $\supp(X)$ has dominated splitting. Since $X$ does not have dominated splitting, from Proposition \ref{Xnotds_thencd} there exists $\beta\sp*$ such that for all $1>\beta_1>\max(b,\beta\sp*)$ the property $P(\beta_1)$ holds, that is a contradiction, hence as was claimed, the property $P(\beta_0)$ holds.

\smallskip

Proposition \ref{pad_imply_adcrit} imply that $\crit\beta\neq\emptyset$ for each $(\beta_0,\beta_0)\ge\beta$.

%
%
%
%
%
%
\end{proof}

\subsection{Critical Pair}

Now we work to proof the opposite direction of the Main Theorem. For this, we use the fact that for every critical point, there exists a critical value intrinsically linked with him. This is the notion of critical pair that we introduce in the following paragraph.

\smallskip

\begin{defi}\label{critical_pair} Let $\beta\in\Delta$. We say that a pair $(x,y)\in X\times X$ is a
$\beta$--critical pair if:
\begin{enumerate}
\item $x\in \crit\beta$, with critical direction $\xi$,
\item $y\in \cv\beta$, with critical direction $\varpi$,
\item there exist a sequence of positive integer $l_k$ such that
\[f^{l_k}(x)\rightarrow y\textrm{\, \,  and \, \, }
M^{l_k}\xi\rightarrow \varpi.\]
\end{enumerate}
\end{defi}

\smallskip

It follows directly of the previous definition the following Proposition.

\smallskip

\begin{pr}\label{ds_not_cp}
If $X$ has dominated splitting, then $X$ does not have a
$\beta$--critical pair.
\end{pr}
\begin{proof}
If $A$ have dominated splitting $TX=E\oplus F$, then the angle of the invariant splitting is great of some $\alpha>0$. If $(x,y)$ is critical pair, the direction $F_x$ is defined by $\xi$, and $E_y$ is defined by $\varpi$, but by the third condition on the previous
definition we have that $M^{l_k}(F_x)\rightarrow E_y$, and this say that $F_y=E_y$; a contradiction.
\end{proof}

\smallskip

The following Proposition, related each $\beta$--critical
point with a $\beta$--critical value.

\smallskip

\begin{pr}\label{cr_pair}
Let $\beta\in\Delta$. Then for every $\beta$--critical point $x$, there exists a $\beta$--critical value $y$, such that the pair $(x,y)$ is a $\beta$--critical pair.
\end{pr}

\smallskip

\begin{rem}
Given a critical point $x$, the critical value  $y$ is not, a priori, uniquely defined, can be occurs that for different critical values $y$ and $y'$, makes $(x,y)$ and $(x,y')$ critical pairs.
\end{rem}

\smallskip

Only remains to proof the Proposition \ref{cr_pair}. For this, we need of the next lemma.

\smallskip

\begin{lm}\label{lem_tec}
Let $X$ be a $b$--asymptotically dissipative such that every $f$--invariant measure is partially hyperbolic.
Let $1>\beta_1,\beta_2>b$, let $x\in X$, and let $\xi_x\in\Cb_x$. If $g(n,\xi_x)\ge\beta_2\sp n$ for each $n\ge0$, then for every $k>0$ there exists $m_k$ such that
\[g(n,M\sp{m_k}\xi_x)\ge\beta_1\sp{-n}\]
for all $0\le n\le k$.
\end{lm}
\begin{proof}Let $\beta_1>\beta_0>b$. The proof goes through the following claim.

\smallskip

\noindent
{\bf Claim.} {\it For every $n'$ and $m'$ positive, there exist $l\ge n'$ and $m\ge m'$ such that}
\[g(l,M\sp m\xi_x)\ge\beta_0\sp{-l}.\]
\begin{proof}[{\it Proof of Claim.}]
If the previous assertion not holds, we have that there exist $n', m'$ such that for all $l\ge n'$ and $m\ge m'$,
$g(l,M\sp m\xi_x)<\beta_0\sp{-l}$. Since $g(n,\xi_x)\ge\beta_2\sp n$, for each $n\ge n'$ we are in the hypothesis of the Criteria of Negative Exponents, that is a contradiction.
\end{proof}

\smallskip

From the preceding, we conclude that for each $n'\ge1$, there exist $l\ge n'$ and $m\ge1$ such that $g(l,M\sp m\xi_x)\ge \beta_0\sp{-l}$.

\smallskip

Now, we will apply the Corollary \ref{ap_p_lemma}. Let $k>0$ fixed. Let $\gamma_0=\beta_0\sp{-1}$ and $\gamma_1=\beta_1\sp{-1}$, and let $n_0$ and $\delta_0$ be the numbers given by this Corollary. We take $n'>n_0$ such that $n'\delta_0-1>k$. Since $g(l,M\sp m\xi_x)\ge \beta_0\sp{-l}$, there exists $0\le j<l$ such that $l-j\ge l\delta_0-1\ge n'\delta_0-1>k$ and that
\[g(i,M\sp{m+j}\xi_x)\ge\beta_1\sp{-i}\]for each $0<i\le l-j$. Therefore taking $m_k=m+j$, we obtain that
\[g(n,M\sp{m_k}\xi_x)\ge\beta_1\sp{-n}\]for each $0\le n\le k$, as asserted.
\end{proof}

\smallskip

\begin{proof}[{\bf Proof of Proposition \ref{cr_pair}:}]
Let $x\in\crit\beta$ with critical direction $\xi_x$. Since $\beta_+>b$ and $g(n,\xi_x)\ge\beta_+\sp n$, from Lemma \ref{lem_tec} it follows that for each positive $k$, there exists $m_k$ such that
\[g(n,M\sp{m_k}\xi_x)\ge\beta_-\sp{-n}\]
for all $0\le n\le k$.

\smallskip

On the other hand, we have that for each $n\ge0$
\[g(-n,\xi_x)\ge\beta_-\sp{-n}>1>\beta_+\sp n.\]
We define $n_0=-m_k$, $n_1=0$, $\delta_\pm=-\log(\beta_\pm)$ and
\[
a_n=\begin{cases}
\log(g(-n,\xi_x)) & n_0-k\le n<+\infty\\
(n-n_0)\delta_-+a_{n_0} & n<n_0-k
\end{cases}.
\]

\smallskip

It is not difficult to see that we are in the hypothesis of Lemma \ref{num_sub_lemma}, and hence there exists $-l_k\in[-m_k,0]$ such that $a_{\pm n-l_k}-a_{-l_k}\ge\pm n\delta_\pm$ for all $n\ge0$. From the construction of sequence $(a_n)_n$ we can conclude that $g(\mp n,M\sp{l_k}\xi_x)\ge \beta_\pm\sp{\pm n}$ for all $0\le n\le k$.

\smallskip

From compactness, and taking a subsequence if necessary, there exist $y\in X$ and $\varpi\in\Cb_y$ such that
\[f^{l_k}(x)\rightarrow y\textrm{\, \,  and \, \, }
M^{l_k}\xi\rightarrow \varpi\]
and therefore $y\in\cv\beta$.
\end{proof}

\smallskip

Now we will conclude the proof of Main Theorem.

\smallskip

\begin{proof}[{\bf Proof of Main Theorem:}]\label{crit_not_ds} {\it
If $\crit\beta\neq\emptyset$ with $\beta_+>b$, then $X$ does not have Dominated Splitting:} \newline If there exist critical point, then there exist a critical pair, so by Proposition \ref{ds_not_cp}, $X$ does not have dominated splitting.
\end{proof}

\section{Properties of the Critical Point}\label{sec:6}

In this section we explain in details a series of properties referring to critical points. Initially, we state notions, that represent equivalences with the notion of critical points.

\smallskip

\begin{defi}

We say that $x\in X$ is a {\bf $\beta$--post--critical point of order $N\in\Z\sp+$}, if there exists $n\in\Z$ with $|n|\le N$ such that $f\sp n(x)\in\crit{\beta}$.
%
\end{defi}

\smallskip

Note that, a post--critical point of order $0$, is a critical point.

\smallskip

In the definition above, when $n$ is negative, it is more natural to replace the word ``post--critical'' by ``pre--critical''. To avoid overloading the language, we choose the terminology post--critical given the sense that this point is an iterate (positive or negative) of a critical point.

\smallskip

Our following result, explain that really we have only post--critical points.

\smallskip

\begin{teo}\label{teo:crit-1}
If $x\in X$ be a $\beta$--critical point at the times $(n_-,n_+)$, then $x$ is a $\beta$--post--critical point of order $N=\max\{n_+,-n_-\}$.
\end{teo}

\smallskip

To prove this Theorem, we need of the Lemma \ref{num_sub_lemma}.

\smallskip

\begin{proof}[{\bf Proof of Lemma \ref{num_sub_lemma}}]
The proof goes through the following claim.

\smallskip

\noindent
{\bf Claim:\, }{\it Let $(a_n)_{n\in\Z}\subset \R$ be a sequence and there exist $n_0\le0\le n_1$ and $-\infty<\delta_+\le\delta_-<0$ such that:
\begin{itemize}
\item[{\it i)}] $a_n\ge(n-n_0)\delta_-$, for all $n\le n_0$,
\item[{\it ii)}] $a_n\ge(n-n_1)\delta_+$, for all $n\ge n_1$.
\end{itemize}
Then there exists ${N}\in[n_0,n_1]$ such that $a_{\pm n+N}-a_N\ge \pm n\delta_\pm$ for all $n\ge0$.}
\begin{proof}[{\bf Proof of Claim:}]
Let $h:\R\to\R$ by the function defined by
\[
h(x)=\begin{cases}
\, \, a_n &, x=n\in\Z\\
\, \, (x-n)a_{n+1}+(1-(x-n))a_n &, n<x<n+1
\end{cases},
\]
then $h$ is a continuous polygonal function with vertices on $\Z$. Take
\[d_-=\sup\{d\in(-\infty,-n_0\delta_-]\,:\, h(x)>x\delta_-+d,\, \forall\, x\in[n_0,0]\}\]
and
\[d_+=\sup\{d\in(-\infty,0]\,:\, h(x)>x\delta_++d,\, \forall\,  x\in[0,n_1]\}.\]
It follows that the $\graph(h|_{[n_0,0]})$ is tangent to the line $L_-=\{y=x\delta_-+d_-\}$ and $\graph(h|_{[0,n_1]})$ is tangent to the line $L_+=\{y=x\delta_++d_+\}$. Since the graph of $h$ is a polygonal, the set of tangency between the graph of $h$ and the lines $L_-$ and $L_+$, that is,
\[T=\{x\le0\,:\,(x,h(x))\in L_-\}\cup\{x\ge0\,:\,(x,h(x))\in L_+\}\]
satisfies $\partial T\subset\Z$.

\smallskip

On one hand, we assume that $d_-\le d_+$. Hence if we take $N\in[n_0,0]$ the largest integer in the tangency $T$, then it is easy to see that
\[
h(x)\ge\begin{cases}
(x-n_0)\delta_-\ge x\delta_-+d_- & ,\, x\le n_0\\
x\delta_-+d_- & ,\, n_0\le x\le N\\
x\delta_-+d_-=(N+n)\delta_-+d_-\ge n\delta_++N\delta_-+d_- & ,\, N\le x=N+n\le0,\, n\ge0\\
x\delta_++d_+\ge x\delta_++d_- & ,\, 0\le x\le n_1\\
(x-n_1)\delta_+\ge x\delta_+\ge x\delta_++d_- & ,\, n_1\le x
\end{cases}.
\]
From the choice of $N\le0$ we have that $d_-=h(N)-N\delta_-$, hence
\begin{equation}\label{eq:num_lemm}h(n+N)\ge\begin{cases}
(n+N)\delta_-+d_-=n\delta_-+h(N) & ,\, n\le 0\\
(n+N)\delta_++d_-\ge n\delta_++h(N) & ,\, n\ge0
\end{cases}.\end{equation}
On the other hand, when $d_+\le d_-$, we take $N\in[0,n_1]$ the lowest integer in the tangency $T$. Also we consider the line $L_0=\{y=x\delta_-+h(N)-N\delta_-\}$ and $\delta_0=L_0(0)=h(N)-N\delta_-$. Since $\delta_+\le\delta_-<0$, we have that
\[\delta_0\le L_+(0)=d_+\le d_-\le -n_0\delta_-<0,\]
it follows that
\[
h(x)\ge\begin{cases}
(x-n_0)\delta_-\ge x\delta_-+\delta_0 & ,\, x\le n_0\\
x\delta_-+d_-\ge x\delta_-+\delta_0 & ,\, n_0\le x\le 0\\
x\delta_++d_+\ge x\delta_-+\delta_0 & ,\, 0\le x\le N\\
x\delta_++d_+ & ,\, N\le x\le n_1\\
(x-n_1)\delta_+\ge x\delta_+\ge x\delta_++d_+ & ,\, n_1\le x
\end{cases}.
\]
Since $d_+=h(N)-N\delta_+$ and from the choice of $\delta_0$ it is easy to see that we have the same inequality has in inequality (\ref{eq:num_lemm}). Since $a_{n+N}-a_N=h(n+N)-h(N)$ the lemma follows.
\end{proof}

We define $a=\min\{a_{n_0},a_{n_1}\}$. Then we have that
\begin{itemize}
\item[{\it a')}] $a_n\ge(n-n_0)\delta_-+a_{n_0}\ge (n-n_0)\delta_-+a$, for all $n\le n_0$,
\item[{\it b')}] $a_n\ge(n-n_1)\delta_++a_{n_1}\ge(n-n_1)\delta_++a$, for all $n\ge n_1$.
\end{itemize}
Therefore the sequence $(b_n)_\Z$ where $b_n=a_n-a$ satisfies hypothesis of the Claim. Since $a_{n+N}-a_N=b_{n+N}-b_N$, we prove the corollary.
\end{proof}

\smallskip

\begin{proof}[{\bf Proof of Theorem \ref{teo:crit-1}}]
From Corollary \ref{num_sub_lemma} taking $a_n=g(n,\xi_x)$ and $\delta_\pm=\log(\beta_\pm)$ there exists $|n|\leq N$ such that
\[a_{\pm k+n}-a_n\ge \pm k\delta_{\pm}\]
for all $k\ge0$. Then
\[\beta_{\pm}\sp{\pm k}\le g(k+n,M\sp n\xi_x)\cdot g(n,\xi)\sp{-1}=g(k,M\sp n\xi_x)\cdot g(n,\xi)\cdot g(n,\xi)\sp{-1}=g(k,M\sp n\xi_x),\]
that is, $f\sp n(x)$ is a $\beta$--critical point.
\end{proof}

\subsection{Critical Points versus Block of Domination}\label{subs:block}

In the seminal work of Pujals and Rodriguez Hertz (vide \cite{P-RH}), critical point are defined as a point such that $x\in H\sp-(\beta)$ and $f\sp n(x)\notin\Hc\sp-(\beta)$ for every $n\ge1$, where $0<\beta<1$. This definition is coherent with the characterization given by Lemma \ref{dom_and_constant}: {\it Let $K\subset X$ be a $f$--invariant set. If  $K\subset\Hc\sp-(\beta)$, then $K$ has dominated splitting}. Moreover, if every point of $K$ have an infinity (for the future) of iterates in $\Hc\sp-(\beta)$, then $X$ have dominated splitting. Then it is necessary, to think in an obstruction for domination, that the positive orbit of a point not is contained in $\Hc\sp-(\beta)$.

In this subsection, we relate the notions of $\beta$--critical point at the time $(n_-,n_+)$ and $\beta$--post--critical point, with the block of domination.

\begin{pr}\label{Hc_prop}For $i=1,2$, let $\beta_i>0$. Assume that $x\in H^-(\beta_1)$ with critical direction $\xi_x$ and let $l>0$. Then the following
statements occur:
\begin{itemize}
\item[{\it i)}]If $f^l(x)\in H^-(\beta_2)$ (or in
$\Hc^-(\beta_2)$), then $\xi_{f^l(x)}=M^l(\xi_x)$, that is, $M\sp l(\xi_x)$ is the critical direction for $f\sp
l(x)$.
\item[{\it ii)}]If $g(l,\xi_x)>\beta_2^{-l}$ then
$f^l(x)\notin\Hc^-(\beta_2)$.
\item[{\it iii)}]If $g(l,\xi_x)\geq\beta_2^l$ then
$f^l(x)\notin\Hc^-(\beta_2)$.
\end{itemize}
\end{pr}
\begin{proof} Let $v_x$ and $v_{f\sp l(x)}$ be unitary vectors such that $[v_x]=\xi_x$, $[v_{f\sp l(x)}]=\xi_{f\sp l(x)}$ and $\measuredangle(v_{f^l(x)},A^l_xv_x)\leq\pi/2$ is minimal. To prove {\it (i)}, we claim that $\measuredangle(v_{f^l(x)},A^l_xv_x)=0$. In fact, if $\measuredangle(v_{f^l(x)},A^l_xv_x)>0$
then for every $n<-l$ we have that
\[g(n,M^l(\xi_x))=g(n+l,\xi_x)g(-l,M^l(\xi_x))\geq\beta_1^{n+l}g(-l,M^l(\xi_x)).\]
Defining $\beta=\max(\beta_1,\beta_2)$ we have that for any $n\le0$
\[\begin{aligned}g({-l},M^l(\xi_x))\beta^{2n+l}\leq
g({-l},M^l(\xi_x))\beta_1^{n+l}\beta_2\sp{n}&\leq g({n},M^l(\xi_{x}))g({n},\xi_{f^l(x)})\\&=\left(\frac{
\sin(\measuredangle(A^{n}_{f\sp l(x)}v_{f^l(x)} , A^{n+l}_{f\sp
l(x)}v_x))}{\sin(\measuredangle(v_{f^l(x)},A^l_xv_x))}
\right)^2\\&<\frac{1}{\sin(\measuredangle(v_{f^l(x)},A^l_xv_x))},
\end{aligned}\]
that is a contradiction.

To prove assertion {\it (ii)}, suppose that $f^l(x)\in\Hc^-(\beta_2)$. Hence it follows from {\it (i)} that for every $n\le0$, $g(n,M^l(\xi_x))>\beta_2^n$. In
particular $g(-l,M^l(\xi_x))>\beta_2^{-l}$, hence
\[\beta_2^{l}<g(-l,M^l(\xi_x))=\frac{1}{g(l,\xi_x)}<\beta_2\sp l\]
that is a contradiction.

To prove assertion {\it (iii)}, suppose
that $f^l(x)\in\Hc^-(\beta_2)$. Then arguing as in the previous assertion, we have that
\[\beta_2^{-l}<g({-l},M^l(\xi_x))=\frac{1}{g(l,\xi_x)}\leq\frac{1}{\beta_2^l}=\beta_2\sp{-l}\]
also a contradiction.
\end{proof}

Critical points at the times $(n_-,n_+)$ can be related with a version in terms of block of dominations.

\begin{teo}\label{teo:01}
If $x\in X$ is a $\beta$--critical point at the times $(n_-,n_+)$, then
\begin{itemize}
\item[{\it a)}] $f\sp{n_-}(x)\in H\sp-(\beta_-)$.
\item[{\it b)}] $f\sp{n+n_+}(x)\notin \Hc\sp-(\beta_+)$ for each $n\geq 1$.
\end{itemize}
On the other hand, if $\beta=(\beta_0,\beta_0)$ and $x$ satisfies the item (a) and (b), then $x$ is a $\beta$--post--critical point of order $n_+-n_-$.

\end{teo}
\begin{proof}
If $x$ is a $\beta$--critical point at the times $(n_-,n_+)$, then there exist a direction $\xi_x$ such that for every $n\ge0$ we have that $g(\pm n,M\sp{n_\pm}\xi_x)\ge\beta_{\pm}\sp{\pm n}$. Since
\[g(-n,M\sp{n_-}\xi_x)\ge\beta_-\sp{-n}\]
for all $n\ge0$ it follows that $f\sp{n_-}(x)\in H\sp-(\beta_-)$.

On the other hand, taking
\[\beta_0=\max\{\beta_-,\sqrt[k]{g(k,M\sp{n_+-k}\xi_x)}\, :\, k=1,\ldots,n_+-n_-\}\]
we have that $f\sp{n_+}(x)\in H\sp-(\beta_0)$. Since $g(n,M\sp{n_+}\xi_x)\ge\beta_+\sp{n}$ for each $n\ge0$, it follows from Proposition \ref{Hc_prop} that $f\sp{n+n_+}(x)\notin\Hc\sp-(\beta_+)$.

Now suppose that $x$ satisfy the items {\it (a)} and ${\it (b)}$ with $\beta=(\beta_0,\beta_0)$ and denote the critical direction of $x$ by $\xi_x$. Without loss of generality we can assume that $n_-=0$. Let $0\le l\le n_+$ maximal with the property $f\sp l(x)\in H\sp-(\beta_0)$. We claim that, for every $n\ge1$,
\[g(-n,M\sp{l+n}\xi_x)<\beta_0\sp{-n}.\]
Assuming the preceding is true, we have that $g(n,M\sp{l}\xi_x)\ge\beta_0\sp n$ for every $n\ge0$ which prove the Theorem. It remains to prove our claim.

First, we suppose that $g(-1,M\sp{l+1}\xi_x)\ge\beta_0$. Then for every $k\ge1$ we have that
\[g(-k,M\sp{l+1}\xi_x)=g(-(k-1),M\sp{l}\xi_x)\cdot g(-1,M\sp{l+1}\xi_x)\ge\beta_0\sp{-(k-1)}\beta_0\sp{-1}=\beta_0\sp{-k}\]
that implies that $f\sp{l+1}(x)\in H\sp-(\beta_0)$, and this contradicts the maximality of  $l$.

Now, we assume that our claim is true for each $0<n<m$.

Finally, we suppose that $g(-m,M\sp{l+m}\xi_x)\ge\beta_0\sp{-m}$. Then for each $0<k\leq m$ we have that
\[
\begin{aligned}
\beta_0\sp{-m}\le g(-m,M\sp{l+m}\xi_x)&=g(-k,M\sp{l+m}\xi_x)\cdot g(-(m-k),M\sp{l+m-k}\xi_x)\\
&<g(-k,M\sp{l+m}\xi_x)\beta_0\sp{m-k}
\end{aligned}
\]
hence $g(-k,M\sp{l+m}\xi_x)\ge\beta_0\sp{-k}$. Similarly, if $k>m$, we have that
\[g(-k,M\sp{l+m}\xi_x)=g(-m,M\sp{l+m}\xi_x)\cdot g(-(k-m),M\sp{l}\xi_x)>\beta_0\sp{-m}\beta_0\sp{-(k-m)}=\beta_0\sp{-k},\]
that is, for all $k\ge0$ we have that $g(-k,M\sp{l+m}\xi_x)\ge \beta_0\sp{-k}$ that is a contradiction with the maximality of $l$.

\end{proof}

\smallskip
\begin{rem}
From the preceding Theorem, we have that $x$ is a $(\beta_0,\beta_0)$--critical point, if and only if $x\in H\sp-(\beta_0)$ and $f\sp n(x)\notin \Hc\sp-(\beta_0)$, that is the original definition in \cite{P-RH}.
\end{rem}
\smallskip

\subsection{The Critical Set}

In this section we explain the main properties of $\crit\beta$. These properties justify the notion of critical point, show how the notion of critical point is an intrinsic notion of the dynamics, and highlight its meaning.

\bigskip

\noindent
$\bullet$ {\bf Compactness:}\, {\it We recall that the set of critical point is a compact set}.

\smallskip

\noindent
In fact, let $(x_k)_\N\subset\crit\beta$ such that $x_k\rightarrow x$, and denote his critical directions by $\xi_k$. By passing to subsequence if necessary, there exist a direction $\xi\in\Cb_x$ such that $\xi_k\rightarrow\xi$. Since for each $n\ge0$ we have that $g(\pm n,\xi_k)\ge\beta_\pm\sp{\pm n}$, taking $k$ goes to infinity we conclude that $x\in\crit\beta$.

\bigskip

\noindent
$\bullet$ {\bf Distinguished Critical Point:}\, {\it  We assert that, if $\beta\in\Delta$ with $\beta_+>b$, then in the orbit of a critical point $x$, there exist a critical point positively maximal}.

\smallskip

\noindent
In others words, if $x\in\crit\beta$ there exist $n_0\ge0$ such that $f\sp n(f\sp{n_0}(x))\notin\crit\beta$ for each $n\ge1$ (We call the maximal element in the orbit of a critical point, by {\bf distinguished critical  point}).

In fact, suppose by contradiction that there exists $(n_k)\nearrow\infty$ such that $f\sp{n_k}(x)\in\crit\beta$. Without loss of generality, we can take $n_k\rightarrow\infty$ such that the limit
\[\mu=\lim_{k\rightarrow\infty}\frac{1}{n_k}\sum_{i=0}^{n_k-1}\delta_{
M^i(\xi_x)}\]
there exists. Denote the projection in the first variable of $\mu $ by $\mu'$.

On the other hand, since $g(-n_k,M\sp{n_k}\xi)\ge\beta_-\sp{-n_k}$ then the inequality \begin{equation}\label{eq:29}
\beta_+\sp{n_k}\le g(n_k,\xi)\le\beta_-\sp{n_k}<1<\beta_-\sp{-n_k}\end{equation} holds.

We recall that
\[\supp(\mu)=\bigcap_{k\ge1}\overline{\{(f\sp{n_s}(x),M\sp{n_s}\xi_x)\, :\, s\ge k\}}.\]
Hence, arguing as in the proof of Criteria of Negative Exponent, and since the equation (\ref{eq:29}) holds, we can conclude that for every $(z,w)$ with $x\in\mathcal{R}(A,\mu')$ the limit $I(z,w)$ defined in the equation (\ref{eq:02}) satisfy
\[\log(\beta_+)\le I(z,w)\le-\log(\beta_-).\]

Since $\beta_+,\beta_->b$ and $f$ has no attractor, then working in the same way as in Criteria of Negative Exponent we obtain a contradiction.

\bigskip

\noindent
$\bullet$ {\bf Change of Metric:}\,  Let $(\cdot|\cdot)_i$ be a hermitian metric in $TX$, where $i=0,1$. Denotes his spherical metrics related with them by $||\cdot||_i$ (see Appendix for details), and denote the set of critical points by ${\rm Crit}_i(\beta)$.
\begin{center}
{\it We claim that there exist a positive integer $N$ such that}
\[{\rm Crit}_0(\beta)\subset\bigcup_{j=-N}\sp N f\sp j({\rm Crit}_1(\beta)).\]
{\it In other words, every $\beta$--critical point to $g_0$, is a $\beta$--post--critical\\ point for $g_1$ of order $N$.}
\end{center}

In fact, let $\alpha>0$ such that $\alpha\sp{-1}||\cdot||_1\le ||\cdot||_0\le \alpha||\cdot||_1$. We recall from Definition \ref{def:g}, equation (\ref{eq:03}), and Definition \ref{defi:g_for_cocyles}, that for $\xi\in\Cb_x$, $n\in\Z$ and for every $w\in T_\xi\Cb_x$ we have that
\[g_i(n,\xi)=||(M\sp n)'(\xi)||_{i}=\frac{||(M\sp n)'(\xi)w||_{i,R\sp n(\xi)}}{||w||_{i,\xi}}.\]
Replacing the previous equation we have
\[\frac{g_1(n,\xi)}{g_0(n,\xi)}=\frac{||(M\sp n)'(\xi)w||_{1,R\sp n(\xi)}}{||(M\sp n)'(\xi)w||_{0,R\sp n(\xi)}}\cdot\frac{||w||_{0,\xi}}{||w||_{1,\xi}}\]
and hence we conclude that $\alpha\sp{-2}\le g_1/g_0\le\alpha\sp2$. It is easy to see that if $\alpha\le1$ then every $\beta$--critical point to $g_0$, is a $\beta$--critical point for $g_1$ and reciprocally. Then we may assume that $\alpha>1$.

Let $x\in {\rm Crit}_0(\beta)$. Let $\alpha_0=2\log(\alpha)$, let $\delta_\pm=\log(\beta_\pm)$ and let $a_n=\log(g_1(n,\xi))$. If we take $L_\pm(x,d)=x\delta\pm+d$ then for each $n\ge0$ we have that
\[\log(g_0(\pm n,\xi))+\alpha_0\ge\log(g_1(\pm n,\xi))\ge\log(g_0(\pm n,\xi))-\alpha_0\ge L_\pm(\pm n,-\alpha_0).\]

Take $\alpha_1=-\alpha_0\delta_+/\delta_-$. Denote the function floor and ceiling by $\lfloor\cdot\rfloor$ and $\lceil\cdot\rceil$ respectively. Since the lines $L_-(x,-\alpha_0)$ and $L_+(x,\alpha_1)$ is the point $\{(\alpha_0/\delta_-,0)\}$, then there exist
\[d_+=\sup\{d\in[\alpha_1,\alpha_0]\, :\, a_n\ge n\delta_++d,\, \, \forall\, \, n\ge\lceil-d/\delta_+\rceil\}\]
and
\[d_-=\sup\{d\in[-\alpha_0,-d_+/\delta_+]\, :\, a_n\ge n\delta_-+d,\, \, \forall\, \, n\le\lfloor-d/\delta_-\rfloor\}.\]

Define $n_+=\lceil-d_+/\delta_+\rceil$ and $n_-=\lfloor-d_-/\delta_-\rfloor$. From the choice of the constants, it is not difficult to see that
\[a_n\ge(n-n_+)\delta_+, \, \, \textrm{where } n\ge n_+\, \, \textrm{ and }\, \,a_n\ge(n-n_-)\delta_-, \, \, \textrm{where } n\le n_-.\]

It follows that we are in the hypothesis of the Claim  in the proof of Lemma \ref{num_sub_lemma}, then there exists $N'\in[n_-,n_+]$ such that $a_{\pm n+N'}-a_{N'}\ge\pm\delta_\pm$ for all $n\ge0$. This we conclude that $f\sp{N'}(x)\in{\rm Crit}_1(\beta)$. Finally, since
\[\left\lfloor\frac{\alpha_0}{\delta_-}\right\rfloor\le n_-\le n_+\le\left\lceil-\frac{\alpha_0}{\delta_+}\right\rceil\]
we conclude that taking $N=\max\left\{\left|\lfloor\frac{\alpha_0}{\delta_-}\rfloor\right|,
\left|\lceil-\frac{\alpha_0}{\delta_+}\rceil\right|\right\}$, we have our assertion.

\bigskip

$\bullet$ {\bf Conjugated Cocycles:}\, As the notion of dominated splitting is invariant by conjugation, we can expect a similar property  to the notion of critical point.
More precisely, we state:
\begin{lm}
Let $A$ and $B$ be two conjugated lineal cocycles over $TX$. Then $\beta$--critical points of $A$ becomes $\beta$--post--critical points of $B$.
\end{lm}
\begin{proof}
In fact, let $A=(f_0,A_*)$,\, $B=(f_1,B_*)$ and
$H=(h,H_*)$ be linear cocycles and $M$,\, $N$ and $L$, the
respective projective cocycles related with them. Assume that $H\circ A=B\circ H$. Denote the norm of the multiplier related with $A$ (resp. $B$) by $g_0$ (resp. $g_1$).

Then  we have that
\[
\begin{aligned}
g_0(n,\xi)=||(M\sp n)'(\xi)||&\le||(L\sp{-1})'(N\sp n(L(\xi)))||\cdot||L'(\xi)||\cdot||(N\sp n)'(L(\xi))||\\
&=||(L\sp{-1})'(N\sp n(L(\xi)))||\cdot||L'(\xi)||\cdot g_1(n,L(\xi)).
\end{aligned}
\]
Hence taking $c_\pm=\sup\{||(L\sp{\pm 1})'(\varpi)||\, : \, \varpi\in\P(X)\}$ and $c=c_-\cdot c_+$ we conclude that $g_0(n,\xi)\le c g_1(n,L(\xi))$. Of similar way, we can conclude that $g_1(n,\xi)\le cg_0(n,L\sp{-1}(\xi))$ and therefore $c\sp{-1}\le g_1(n,L(\xi))/g_0(n,\xi)\le c$.

Arguing as in the preceding item, we can conclude that if $x$ is a critical point for $A$, then $h(x)$ is a post--critical point of order bounded.
\end{proof}

\subsection{Proof of Theorem \ref{thm_A}}\label{proofteoblock}
\begin{proof}
Let $1>\beta>b$ and $1\le\gamma<b\sp{-1}\beta$. We can choice $\alpha$ such that $b<\alpha<\gamma\sp{-1}\beta$.
Let $\nu$ be a $f$--invariant measure and let
$x\in\mathcal{R}(A,\nu)$ (see, Subsection 2.2), then
\[\lim_{n\rightarrow+\infty}\frac{1}{n}\log\left(\phantom{\overline{A}}\hspace{-8pt}g(n,E_x)
\right)=\lambda^+(x)-\lambda^-(x)\geq-\lambda^-(x)\geq-\log(b).\]
We consider \[1>\alpha>c>b\] arbitrarily but fixes. It follows that
for $m$ large enough we have that \[g(m,E_x)\geq c^{-m}.\] From Pliss's Lemma, there exists a sequence $(m_k)_k\nearrow\infty$ satisfying
\[g(n,M^{m_k}(E_x))\geq\alpha^{-n}>\gamma\sp n\beta\sp{-n}>\gamma\beta\sp{-n}\, \, \,\textrm{for every}\, \, \, n\geq1.\]
It follows that $\gamma H^+(\beta)$ is a not empty set. Moreover, $\gamma H^+(\beta)$ contains all accumulation points of the set $(f^{m_k}(x))_k$, which critical directions is an accumulation point of the set $(M^{m_k}(E_x))_k$.

Arguing in the same way, for $x\in\mathcal{R}(A,\nu)$ we have that
\[\lim_{n\rightarrow+\infty}\frac{1}{n}\log
\left(\phantom{\overline{A}}\hspace{-8pt}g({-n},F_x)
\right)=\lambda^+(x)-\lambda^-(x)\geq\lambda^-(x)\geq-\log(b),\]
and we can find a sequence $(n_k)_k\nearrow\infty$ such that
\[g({-n},M^{-n_k}(F_x))\geq\beta\sp{-n}>\gamma\sp{-1}\beta\sp{-n}\, \, \,\textrm{for every}\, \, \, n\geq1,\]
and conclude that $H^-(\beta)$ is not empty.

To see the compactness of $\gamma\sp{-1}H\sp+(\beta)$, we take a sequence $(y_n)_n\subset H^+(\beta)$ which critical direction $(\varpi_n)_n$. If $y$ is any accumulation point of $(y_n)_n$, then (taking subsequence if necessary) there exists a direction $\varpi_y$ accumulated by the directions $(\varpi_n)_n$ that satisfy
$g(n,\varpi_x)\geq\gamma\beta\sp{-n}$, then $y\in \gamma H^+(\beta)$.

Finally, let $X^+_0=\cup_{n\in\Z}f^n(\gamma H^+(\beta))$. Note that for any regular point $x\in\mathcal{R}(A,\nu)$, there exists a forward iterate of $x$ in $\gamma H^+(\beta)$, then $\mathcal{R}(A,\nu)\subset X^+_0$ and hence, $X^+_0$ have total measure.
\end{proof}

\subsection{Dynamically Defined Cocycles}

In this subsection, we let $f$ be a biholomorphisms in a two dimensional complex manifolds (for example, a generalized H\'enon map), with a set $X$ compact and $f$--invariant. We also consider the natural cocycle related with $f$, that is $Df_\#=(f,Df)$ the cocycle defined on $TX$ by the function and his derivative. We also assume that $f$ is $b$--asymptotically dissipative and has no attractor.

We recall that in the one dimensional context, both real and complex, critical point are far from hyperbolic set, however they can accumulated by hyperbolic sets.

In the two dimensional context, a similar result holds:

\begin{cor}
Suppose that $X$ does not have dominated splitting, but there exist an $f$--invariant compact set $X'\subset X$ such that $X'$ is hyperbolic and/or has dominated splitting. Then for every $\beta$ with $\beta_+>b$, $\textrm{dist}(X',\crit\beta)>0$.
\end{cor}

In fact, this follows from the Main Theorem, since both $X'$ and $\crit\beta$ are compact and disjoints. Moreover, we also can state:

\begin{lm}
A critical point is not a regular point.
\end{lm}
\begin{proof}
Let $x$ be a $\beta$--critical point with critical direction $\xi$ and $\beta_-\ge\beta_+>b$. We assume that $x$ is a regular. Let $T_x=E\sp+\oplus E\sp-$  a splitting related with the Lyapunov exponents. We also have the inequalities
\[\lambda\sp-\leq\log(b)<0\leq\lambda\sp+,\]
and
\[\lambda\sp--\lambda\sp+\leq\log(b).\]
We assert that the direction related with $\xi$ is the subspace $E\sp+$. In fact, if not, we have
\[\lim_{n\rightarrow-\infty}\frac{1}{n}\log||Df_z\sp
n\upsilon_\xi||=-\lambda\sp-,\]
where $\upsilon_\xi$ is unitary and define the direction
$\xi$. From the previous equation and since that $g(-n,\xi)\geq\beta_-\sp{-n}$ for $n\geq0$, we have that
\[\lim_{n\rightarrow-\infty}\frac{1}{n}\log(g(
n,\xi))=\lambda\sp--\lambda\sp+\geq-\log(\beta_-),\]
and this implies that $b\ge\beta_-\sp{-1}$ that is a contradiction. Now, as $E\sp+$ define the direction $\xi$, we have that
\[\lim_{n\rightarrow\infty}\frac{1}{n}\log(g(n,\xi))=\lambda\sp--\lambda\sp+.\]
Since  $g(n,\xi)\geq\beta_+\sp{n}$ we conclude hat
\[\log(b)\geq\log(\beta_+),\]
that is a contradiction.
\end{proof}


Moreover, critical point are disjoint to {\it ``hyperbolic blocks''}. For example, let $f$ be a H\'enon map. We denote by $\reg\subset J\sp*$ the set of all regular points. Given $C>0$ fixed, consider the set
\[
\begin{aligned}
B(0,C)=\bpl z\in\reg\, : \, & |Df\sp n|E\sp-(z)|\leq
C\exp(n\lambda\sp-(z))\\ 
& \textrm{ and }
|Df\sp{-n}|E\sp+(z)|\leq C\exp(-n\lambda\sp+(z))\bpr.
\end{aligned}
\]
Clearly, this set is closed, and given $l\in\N$ we define the {\it
hyperbolic block} of \linebreak 
large $l$
\[B(l,C)=\cup_{k=-l}\sp{l}f\sp{k}(B(0,C)).\]

Since all point in the hyperbolic block is regular, it is follows from the previous lemma, that a critical point is disjoint to the blocks.

\bigskip

\noindent
$\bullet$ {\bf Tangencies of a Periodic Point Contain  a Critical Point:}\,
Let $f$ be a H\'enon map with $b=|\det(Df)|<1$. Let $p$ be a periodic point of $f$ and $q$ be a point of tangencie between the stable and the unstable manifolds of $p$.
{\it We assert that in the orbit of $\mathcal{O}(p)$ there exist a $\beta$--critical point, when $\beta\in\Delta$ and $b_+> b$}.

In fact, without loss of generality, we can assume that $p$ is a fixed point. We denote the local stable/unstable manifold of $p$ of size $\varepsilon$ by $\www{s}{p}{\varepsilon}$ and $\www{u}{p}{\varepsilon}$ respectively. Let $\lambda\sp s$ and $\lambda\sp u $ the eigenvalues of $Df$ in $p$, then
$b=|\lambda\sp s|\cdot|\lambda\sp u|$.

\smallskip

Note that for each $n\ge0$
\[g({-n},E_p\sp u)=\frac{b\sp{-n}}{|Df\sp{-n}|E_p\sp
u|\sp2}=\left(\frac{|\lambda\sp u|\sp2}{b}\right)\sp n>b\sp{-n}>\beta_-\sp{-n},\]
and that
\[g({n},E_p\sp s)=\frac{b\sp{n}}{|Df\sp{n}|E_p\sp
s|\sp2}=\frac{|\lambda\sp u|\sp n}{|\lambda\sp
s|\sp n}=\left(\frac{|\lambda\sp
u|\sp2}{b}\right)\sp n>b\sp{-n}>1>\beta_+\sp n.\]

We can take $\varepsilon>0$ is small enough such that for every $z\in\www{s}{\varepsilon}{p}$ (resp. $z\in\www{u}{\varepsilon}{p}$) we have that $z\approx p$ and $T_z\www{s}{\varepsilon}{p}\approx E\sp s_p$ (resp. $T_z\www{u}{\varepsilon}{p}\approx E\sp u_p$). We can conclude that for each $z\in\www{u}{\varepsilon}{p}$ (resp. $z\in\www{s}{\varepsilon}{p}$) we have that $g(-n,T_z\www{u}{\varepsilon}{p})\ge \beta_-\sp{-n}$ (resp. $g(n,T_z\www{s}{\varepsilon}{p})\ge \beta_+\sp{n}$) for each $n\ge0$.

Finally, let $q_u$ be the first iterate to the past of $q$ that is inside of $\www{u}{\varepsilon}{p}$ and let $n_+>0$ such that $f\sp{n_+}(q_u)$ is the first iterate to the future of $q$ that is inside of $\www{s}{\varepsilon}{p}$. Since, $q$ is a tangencie point we have that  $Df\sp{n_+}(T_{q_u}\www{u}{\varepsilon}{p})=
T_{f\sp{n_+}(q_u)}\www{s}{\varepsilon}{p}$, hence we conclude that $q_u$ is a $\beta$--critical point at the times $(0,n_+)$. From Theorem \ref{teo:crit-1} we conclude that there exist $0\le n_0\le n_+$ such that $f\sp{n_0}(q_u)$ is a $\beta$--critical point, so we have our assertion.

\subsection{Some Remark on Critical points for H\'enon maps}\label{p_hol_context}

Recall that a polynomial or a rational maps in $\C$, always have critical points. Moreover, the critical points determine the global dynamics of a polynomial. Whit this in mind, we can state: {\it Let $p$ by a polynomial over $\C$ of degree at least two. Then the Julia set $J_p\subset\C$ is hyperbolic, if and only if, $PC(p)\cap J_p=\emptyset$}. Here $PC(p)$ denote the post--critical set defined by
\[PC(p)=\overline{\cup_{n\geq1}p\sp n\left(\bpl z\, :
p'(z)=0\bpr\right)}.\]

Following these result, we wonder if for a complex H\'enon maps, always there exists critical point (even outside of the Julia). Moreover, if these ``critical points'' there exist, we wonder if they determine the global dynamics. In fact, recall that we have proved: {\it The intersection $\textrm{Crit}\cap J=\emptyset$, if and only if $J$ has dominated splitting, where $\textrm{Crit}$ denote the critical set.}

If we think in the words ``critical point'' as the object that represent the obstruction to have dominated splitting (independent of the adopted definition), we can formulate:

\bigskip

\noindent
{\bf Question A:}\label{que09}\, {\it Do always exists critical point in $\Ctwo$?}

\bigskip

\noindent
{\bf Question B:}\label{que10}\, {\it If $K\sp+$ has interior, always exists critical point in $K\sp+$ ?}

\bigskip

We can answer positively the Question B, for a polynomial
automorphisms close to the one dimensional polynomial $p$. Let
\[f_\delta(x,y)=(y,p(y)-\delta x),\]
with $|\delta|$ small. When we refer to critical point of $p$, i.e., $p'(x)=0$ we denote them as one dimensional critical point.

Let us assume that the polynomial $p$ satisfies:
\begin{enumerate}
\item there are not one dimensional critical point in $J_p$,
\item $J_p$ is connected,
\item the filled Julia set $K_p$ has interior.
\end{enumerate}

The item 3, implies that the set $K\sp+_p$ associated with the two dimensional map $f_0:(x,y)\mapsto(y,p(y))$, has non empty interior. In fact, is easy to see that $K\sp+_p=\C\times K_p$. We recall
that since $|\delta|$ small, $f=f_\delta$ is close to $f_0$, hence
$J_f$ is close to the set $J_0=\bpl (y,p(y))\, :\ , y\in J_p \bpr$.

\begin{pr}
Under the previous hypothesis, there are a critical point in the
interior of $K\sp+$.
\end{pr}
\begin{proof}
If there are not critical points in $K\sp+$, then this has dominated
splitting, so $K\sp+$ is foliated by holomorphic stable leaves
\[K\sp+=\sup_{x\in K\sp+}\ww{s}{x}.\]

On the other hand, the map
\[z=(x,y)\hookrightarrow (y,p(y))\hookrightarrow p(y),\]
is holomorphic and the image of $K\sp+_p$ is $K_p$, that is contained
in the $y$-axis. So, for $|\delta|$ small enough, there exists a
holomorphic disc $D$, close to the $y$-axis and transversal to the
stable foliation of $f$ in $K\sp+$.

We define $\pi\sp s$ the projection to $D$, by the stable
foliation. Now we define
\[z\in D\cap K\sp+\hookrightarrow f(z)\hookrightarrow \pi\sp
s(f(z))\in D.\]
Then the map $(\pi\sp s\circ f):D\rightarrow D$ is a
holomorphic one dimensional map. We denote by $pr_2$ the projection in
the second variable. Since $\pi\sp s$ close to $pr_2$ nearby the
Julia set, and $f$ close to $f_0$, then $\pi\sp s\circ f$ is close to
$p$ nearby the Julia set, thus it is follows that $\pi\sp s\circ f$
has degree equal to degree of polynomial $p$.

From the previous observation, there exist $c\in D$ such that $(\pi\sp
s\circ f)'(c)=0$. Now, $Df$ does not have kernel, it is follows
that \[(Df)(T_c D)\subset\, \textrm{Kernel}(D\pi\sp s).\]
Since that $\textrm{Kernel}(D\pi\sp s)=E\sp s$, we conclude that
$(Df)(T_c D)\subset E\sp s$ but this is a contradiction, because
$T_cD {\pitchfork} E\sp s$.
\end{proof}

Therefore, one of the delicate step, is to give a formal definition of critical point in $K\sp+$ in a general context. All of this, in the direction of a possible state of the following:

{\bf Conjecture:}{\it \, If the Julia set there are not $\beta$--critical point, and if the ``critical points'' of the interior of $K\sp+$ not accumulate on $J$ (thinking as in the one--dimensional post--critical point), then $J$ is hyperbolic.}

\appendix
\section{Appendix}

\subsection{Hermitian and Spherical metrics}\label{s:herm_vs_spherical}

This section is devoted to proving the existence of a spherical metric in the Riemann sphere, given an hermitian metric. For this purpose, it is suffices to make this construction in $\Ctwo$.


%
%


The Riemann sphere is the projective space consisting of all 1--dimensional subspaces of $\C^2$ or complex lines. The complex line that through point $v$ is the set $[v]=\{\lambda v\,:\,\lambda\in\C\sp*\}$. Writing the point $[v]$ in homogeneous coordinates this has the form $[z_1:z_2]$ where $v=(v_1,v_2)$. So we obtain that
\[\Cb=\bpl[z_1:z_2]\, :\, (z_1,z_2)\in\C^2\bpr.\]

Each $z\in\C$ is related with $[z_1:z_2]$ if and only if $[z_1:z_2]=[z_1/z_2:1]=[z:1]$. The point at infinity is related with the class $[1:0]$, and we can write $\Cb=\C\cup\{\infty\}$. In this coordinates the {\bf standard spherical metric}
\begin{equation}\label{sph-met}
d\rho=\frac{2|dz|}{1+|z|^2}
\end{equation}
and has constant Gaussian curvature +1.

The previous construction was made under the representation in homogenous coordinates in the canonical base. Now, we will repeat this construction, but considering an arbitrary base, and we will find the relationship between this different representations.

Let $\beta=\{v_1,v_2\}$ be a base of $\Ctwo$ and write $v=w_1v_1+w_2v_2=(w_1,w_2)_\beta$. We write the homogeneous coordinate in the base $\beta$ of the vector $v$ as $[w_1:w_2]_\beta$. Also we relate each $[w_1:w_2]_\beta$ with the point $w\in\C$ if and only if $w=w_1/w_2$. Finally, we denote
\[\Cb_\beta=\bpl[w_1:w_2]_\beta\, :\, w_1v_1+w_2v_2\in\C^2\bpr,\]
and define the {\bf spherical metric in the base $\beta$ on $\Cb_\beta$} by the equation
\[d\rho_\beta=\frac{2|dw|}{1+|w|^2}.\]


On the other hand, let $L$ the linear transformation satisfying $Lv_i=e_i$ whit $i=1,2$ and where $\{e_1,e_2\}$ denote the canonical base. It is not difficult to see that, denoting the M${\ddot{\rm o}}$bius transformation related with $L$ by $N$, then we have that $N(z)=w$.

Let $v_z=(z_1,z_2)$ such that $z=z_1/z_2$. From equation (\ref{sph-met}) it not difficult to see that
\[d\rho=2|dz|\frac{|(v_z|e_2)|\sp2}{(v_z|v_z)}\]where $(\cdot|\cdot)$ denote the standard hermitian metric. Similarly, if $(\cdot|\cdot)_0$ is a hermitian metric in $\Ctwo$ such that $\beta$ is a orthonormal bases, then
\begin{equation}\label{LA-eq}
d\rho_\beta=2|dw|\frac{|(v_w|v_2)_0|\sp2}{(v_w|v_w)_0}
\end{equation}
where $v_w=(w_1,w_2)_\beta$ such that $w=w_1/w_2$.

To justify that the definition given in equation (\ref{LA-eq}) is a good definition, it is necessary to prove that:

\begin{center}
{\bf \textrm{($\dag$)}\, \hspace{3mm}If $\beta$ is a orthonormal base (different of the canonical base)\\ \, \hspace{3mm}in the standard hermitian metric,\\ \, \hspace{3mm}then $d\rho=d\rho_\beta$.}
\end{center}

In fact, if $\beta$ is a orthonormal base standard hermitian metric, then $L$ is an isometry in the hermitian metric and the induced M${\ddot{\rm o}}$bius transformation $N$ is an isometry in the standard spherical metric. Write
\[L=\left(\begin{matrix}
\overline{b}&-a\\
\overline{a}&b\end{matrix}\right)\]
and
\[N(z)=\frac{\overline{b}z-a}{\overline{a}z+b}.\]
Let $v_z=(z_1,z_2)$ with $z=z_1/z_2$. Since $N(z)=w$, then we can take $v_w=(w_1,w_2)_\beta=w_1v_1+w_2v_2=(\overline{b}z-a,\overline{a}z+b)_\beta$.

Using the notation above and the equation (\ref{LA-eq}) we conclude that
\[
\begin{aligned}
d\rho=2|dz|\frac{|(v_z|e_2)|\sp2}{(v_z|v_z)}&
=2|dz|\frac{|z_2|\sp2}{|z_1|\sp2+|z_2|\sp2}\\
&=2|dz|\frac{|z_2|\sp2}{(Lv_z|Lv_z)}\\
&=2|dz|\frac{|z_2|\sp2}{|\overline{b}z_1-az_2|\sp2+
|\overline{a}z_1+bz_2|\sp2}\\
&=2|dz|\frac{1}{|\overline{b}z-a|\sp2+
|\overline{a}z+b|\sp2}\\
&=\frac{2|dz|}{1+\left|\displaystyle\frac{\overline{b}z-a}{\overline{a}z+b}\right|\sp2}\cdot
\left|\frac{1}{(\overline{a}z+b)\sp2}\right|\\
&=\frac{2|dz|}{1+|N(z)|\sp2}\cdot
\left|N\sp\prime(z)\right|\\
&=2\frac{|dw|}{1+|w|\sp2},
\end{aligned}
\]
and note that
\[\frac{|(v_w|v_2)|\sp2}{(v_w|v_w)}=\frac{|w_2|\sp2}{|w_1|\sp2+|w_2|\sp2}
=\frac{1}{1+|w_1/w_2|\sp2}=\frac{1}{1+|w|\sp2}.\]
Hence $(\dag)$ holds.

Finally, equation (\ref{LA-eq}) allows to define {\bf the spherical metric} as an intrinsic object of the hermitian metric (prefixing an orthonormal base, but not depending of this base). With this, we can justify the existence of a spherical metric in a projective bundle in terms of the hermitian metric defined in the fibre bundle.

%

\bigskip

%
%
%


\begin{thebibliography}{XXX}

\bibitem{bls1} E. Bedford; M. Lyubich; J. Smillie.
  {\it Polynomial diffeomorphisms of $C\sp 2$. IV. The measure of   maximal entropy and laminar currents.} Invent. Math.  112 (1993),  no. 1, 77--125.


\bibitem{bs1} E. Bedford; J. Smillie. {\it Polynomial
  diffeomorphisms of $C\sp 2$: currents, equilibrium  measure and hyperbolicity.} Invent. Math.  103  (1991),  no. 1, 69--99.

\bibitem{bs2} E. Bedford; J. Smillie. {\it Polynomial
  diffeomorphisms of $C\sp 2$. II. Stable manifolds and
  recurrence.} J. Amer. Math. Soc.  4  (1991),  no. 4, 657--679.

\bibitem{bs3} E. Bedford; J. Smillie. {\it Polynomial
  diffeomorphisms of $\C\sp 2$. III. Ergodicity, exponents and   entropy of the equilibrium measure.} Math. Ann.  294  (1992),   no. 3, 395--420.
\bibitem{B-C} M. Benedicks; L. Carleson. {\it The dynamics of the H\'enon map.} Annals of Math., 133 (1991), 73--169.

%
%
%
%
\bibitem{conw} J. Conway. {\it Functions of one complex variable. Second edition.} Graduate Texts in Mathematics, 11. Springer--Verlag, New York--Berlin, 1978. xiii+317 pp.

\bibitem{C} S. Crovisier. {\it Sur la notion de Criticalit\'e de Pujals-Rodriguez Hertz.} Preprint.

\bibitem{F} J. E. Forn\ae{}ss. {\it The Julia set of H\'enon maps.} Math. Ann.  334  (2006),  no. 2, 457--464.

\bibitem{fs1} J. E. Forn\ae{}ss; N. Sibony. {\it Complex H\'enon mappings in $C\sp 2$ and Fatou-Bieberbach domains.} Duke Math. J.  65  (1992),  no. 2, 345--380.

\bibitem{fm} S. Friedland; J. Milnor. {\it Dynamical
  properties of plane polynomial automorphisms.} Ergodic Theory
  Dynam. Systems  9  (1989),  no. 1, 67--99.


\bibitem{h} Hubbard, John H.  {\it The H\'enon mapping in the complex domain}.  Chaotic dynamics and fractals (Atlanta, Ga., 1985), 101--111, Notes Rep. Math. Sci. Engrg., 2, Academic Press, Orlando, FL,  1986.


\bibitem{H} D. Husemoller. {\it Fibre bundles.} Third edition. Graduate Texts in Mathematics, 20. Springer-Verlag, New York, 1994. xx+353 pp.

\bibitem{ho1} J. Hubbard; R. Oberste-Vorth. {\it  H\'enon
  mappings in the complex domain. I. The global topology of dynamical space.}  Inst. Hautes Études Sci. Publ. Math.  No.79  (1994), 5--46.

\bibitem{ho2}  Hubbard, John H. ;  Oberste-Vorth, Ralph W.  {\it
H\'enon mappings in the complex domain. II. Projective and inductive
limits of polynomials}.  Real and complex dynamical systems (Hiller\o d, 1993),  89--132, NATO Adv. Sci. Inst. Ser. C Math. Phys. Sci., 464, Kluwer Acad. Publ., Dordrecht,  1995.

%
%
%
\bibitem{LV} O. Lehto; K. I. Virtanen. {\it  Quasiconformal mappings in the plane.} Second edition. Translated from the German by K. W. Lucas. Die Grundlehren der mathematischen Wissenschaften, Band 126. Springer-Verlag, New York-Heidelberg,  1973. viii+258 pp.


\bibitem{ma2} R. Ma\~n\'e. {\it Hyperbolicity, sinks and measure   in one-dimensional dynamics.} Comm. Math. Phys. 100 (1985), no. 4,   495--524.

\bibitem{M} J. Milnor. {\it Dynamics in one complex variable. Third edition.} Annals of Mathematics Studies, 160. Princeton University Press, Princeton, NJ,  2006. viii+304 pp.

%
\bibitem{P-RH} E. R. Pujals; F. Rodriguez Hertz. {\it
  Critical points for surface diffeomorphisms.} J. Mod. Dyn. 1 (2007),  no. 4, 615--648.

\bibitem{ps} E. R. Pujals; M. Sambarino. {\it Homoclinic tangencies and hyperbolicity for surface diffeomorphisms.} Annals of Mathematics, vol. 151, 2000, p. 961--1023.
%
%

\bibitem{V} M. Viana. {\it Multiplicative ergodic theorem of Oseledets (after Ricardo Ma\~n\'e): A proof of Oseledets' Theorem.} Lecture Notes.

%

\bibitem{pancho01}F. Valenzuela--Henriquez. {\it Equivalent conditions for hyperbolicity on partially hyperbolic holomorphic map.}  Indiana Univ. Math. J. Preprint.

\end{thebibliography}
\end{document}